\documentclass[11pt]{amsart}

\usepackage{amsmath}
\usepackage{amssymb}

\allowdisplaybreaks

\newtheorem{theorem}{Theorem}[section]
\newtheorem{lemma}[theorem]{Lemma}
\newtheorem{proposition}[theorem]{Proposition}
\newtheorem{corollary}[theorem]{Corollary}

\theoremstyle{definition}
\newtheorem{definition}[theorem]{Definition}
\newtheorem{example}[theorem]{Example}
\newtheorem{question}[theorem]{Question}

\theoremstyle{remark}
\newtheorem{remark}[theorem]{Remark}

\newtheorem*{thmA}{Theorem A}
\newtheorem*{thmB}{Theorem B}
\newtheorem*{thmC}{Theorem C}

\newcommand{\D}{\mathrm{D}}                      
\newcommand{\Db}{\mathcal{D}}                    
\newcommand{\per}{\operatorname{per}}
\newcommand{\rmod}{\operatorname{mod}}
\newcommand{\HH}{\mathit{HH}}
\newcommand{\rad}{\operatorname{rad}}
\newcommand{\gldim}{\operatorname{gl.dim}}
\newcommand{\Hom}{\operatorname{Hom}}

\newcommand{\RHom}{\mathbf{R}\!\operatorname{Hom}}
\newcommand{\tr}{\operatorname{tr}}
\newcommand{\ch}{\operatorname{ch}}
\newcommand{\id}{\operatorname{id}}
\newcommand{\lten}{\otimes^{\mathbb{L}}}
\newcommand{\lta}{\otimes^{\mathbb{L}}_{A}}
\newcommand{\ltb}{\otimes^{\mathbb{L}}_{B}}
\newcommand{\kk}{k}
\newcommand{\ZZ}{\mathbb{Z}}
\newcommand{\TT}{\mathbf{TT\text{-}calc}}
\newcommand{\Ak}{\mathbf{A}_{\kk}}
\newcommand{\HB}{\mathbb{H}}
\newcommand{\dimv}{\underline{\dim}\,}
\newcommand{\Kn}{K_0}
\newcommand{\Cox}{\Phi}
\newcommand{\coxpol}{\phi}
\newcommand{\serre}{\omega}
\newcommand{\ar}{\uptau}
\newcommand{\cA}{\mathcal{A}}
\newcommand{\cB}{\mathcal{B}}
\newcommand{\Dcat}{\mathcal{D}}
\newcommand{\PP}{\mathbb{P}}
\newcommand{\QQ}{\mathbb{Q}}
\DeclareMathOperator{\Aut}{Aut}
\newcommand{\uptau}{\tau}

\begin{document}

\title[The Coxeter transformation on the Tamarkin--Tsygan calculus]
{The Coxeter transformation as an automorphism of the Tamarkin--Tsygan calculus}

\author{Marco Armenta}
\address{}
\email{marco.armenta@usherbrooke.ca}



\begin{abstract}
Let $A$ be a finite-dimensional algebra over a field $\kk$. We show that
the Auslander--Reiten bimodule $\ar_A:=\D A[-1]$ is central in the derived
Picard group of $A$ and that, when $\gldim A<\infty$, it induces through
the derived-invariance functor $\HB$ of
\cite{Armenta19,ArmentaKeller17,ArmentaKeller19} a canonical automorphism
$\sigma_A$ of the Tamarkin--Tsygan calculus of $A$; the pair
$(\HB(A),\sigma_A)$ is invariant under derived equivalence. We then compute
both components of $\sigma_A$. On the Hochschild homology of an elementary
algebra, which is concentrated in degree zero, the matrix of $\sigma_A$ in
the basis of idempotent traces is $-C_A^{-1}C_A^{\mathrm{T}}$, so its
characteristic polynomial is the Coxeter polynomial; the enriched calculus
strictly refines both the calculus and the Coxeter polynomial, as the path
algebras of quivers of types $\mathbb{A}_4$ and $\mathbb{D}_4$ show,
although it is not a complete derived invariant, as the smallest
cospectral pair of trees shows. On
Hochschild cohomology we prove that $\sigma_A$ is the identity: the left
and right actions of $\HH^\bullet(A)$ on the bimodule $\D A$ coincide for
every finite-dimensional $A$. This yields a short conceptual proof that
the Nakayama automorphism of a Frobenius algebra acts trivially on
Hochschild cohomology, recovering a recent theorem of Su\'arez-\'Alvarez.
Finally we extend the construction to smooth and proper differential
graded algebras, hence to perfect derived categories of smooth projective
varieties; the enrichment degenerates precisely on Calabi--Yau categories,
and on $\PP^n$ it is governed by the Coxeter polynomial
$(x+(-1)^n)^{n+1}$ of the Beilinson algebra. Happel's trace formula and
de la Pe\~na's cyclotomicity theorem for fractionally Calabi--Yau algebras
become statements internal to the enriched calculus.
\end{abstract}

\maketitle

\section{Introduction}\label{sec:intro}

Two of the oldest families of derived invariants of a finite-dimensional
algebra $A$ of finite global dimension live at opposite ends of the
homological spectrum. At the decategorified end sits the \emph{Coxeter
polynomial} $\coxpol_A$, the characteristic polynomial of the Coxeter
transformation $\Cox_A$, that is, of the automorphism of the Grothendieck
group $\Kn(\per A)$ induced by the derived Auslander--Reiten translation
$\ar=\nu\circ[-1]$ of Happel \cite{Happel88,Happel91}; its derived invariance
is classical, see \cite{delaPena14,LenzingdelaPena08}. At the other end sits
the full Hochschild theory of $A$: Rickard proved that derived equivalences
preserve the Hochschild cohomology algebra \cite{Rickard89a}, Keller that
they preserve its Gerstenhaber bracket \cite{Keller04} and the cyclic theory
\cite{Keller98}, and in \cite{ArmentaKeller17,ArmentaKeller19,Armenta19} it
was proved that the whole \emph{Tamarkin--Tsygan calculus}
\[
  \HB(A) \;=\; \bigl(\HH_\bullet(A),\, \HH^\bullet(A),\, \cup,\, [-,-],\,
  \cap,\, B\bigr)
\]
of Hochschild homology and cohomology, with the cup product, the Gerstenhaber
bracket, the cap product and the Connes differential, is a derived invariant.
More precisely, \cite{Armenta19} constructs a functor
$\HB\colon \Ak\to\TT$, from a category of algebras whose morphisms are
suitable complexes of bimodules to the category of Tamarkin--Tsygan calculi,
which is constant on derived equivalence classes.

Neither invariant determines the other. The thesis \cite{Armenta19} ends
with the observation that the path algebras of the quivers
\[
  Q_A:\quad
  \begin{array}{ccccc}
   2 & & & & 3\\[-2pt]
     & \searrow & & \swarrow & \\[-2pt]
     & & 1 & & \\[-2pt]
     & & \uparrow & & \\[-2pt]
     & & 4 & &
  \end{array}
  \qquad\qquad
  Q_B:\quad 1\longrightarrow 2 \longrightarrow 3 \longrightarrow 4
\]
of Dynkin types $\mathbb{D}_4$ and $\mathbb{A}_4$ have isomorphic
Tamarkin--Tsygan calculi --- every operation degenerates --- while their
Coxeter polynomials
\[
  \coxpol_A(x)=(x+1)^2(x^2-x+1)=x^4+x^3+x+1,
  \qquad
  \coxpol_B(x)=x^4+x^3+x^2+x+1
\]
differ, so that $A$ and $B$ are not derived equivalent --- a fact which
is of course classical \cite{Happel88,Happel91}; see
Remark~\ref{rem:classical} --- and the Tamarkin--Tsygan calculus is not a
complete derived invariant. Conversely,
the Coxeter polynomial is notoriously far from complete
\cite{LenzingdelaPena08}. The purpose of this article is to show that the
two invariants are in fact two shadows of a single, finer one, obtained by
letting the categorical avatar of the Coxeter transformation act on the
calculus itself.

The mechanism is simple, and the point of this paper is precisely that the
machinery of \cite{Armenta19} makes it simple. The Coxeter transformation is
the Grothendieck-group shadow of the Auslander--Reiten translation
$\ar=\nu[-1]$, where $\nu=-\lta \D A$ is the Nakayama functor, a Serre
functor on $\per A$ \cite{BondalKapranov90,Happel88,Happel91}. We prove that
the underlying bimodule complex
\[
  \ar_A \;:=\; \D A[-1]
\]
is a two-sided tilting complex over $A$ whenever $\gldim A<\infty$
(Lemma~\ref{lem:invertible}) and that it commutes with every two-sided
tilting complex (Lemma~\ref{lem:commutation}): if $X$ is a two-sided tilting
complex of $A$-$B$-bimodules, then
\[
  \ar_A\lta X \;\cong\; X\ltb \ar_B
  \qquad\text{in the derived category of $A$-$B$-bimodules.}
\]
Since $\ar_A$ is in particular a morphism $A\to A$ of the category $\Ak$ of
\cite{Armenta19}, the functor $\HB$ produces from it an automorphism of the
calculus. We call
\[
  \sigma_A \;:=\; \HB(\ar_A)\;\in\;\Aut_{\TT}\bigl(\HB(A)\bigr)
\]
the \emph{Coxeter automorphism} of the Tamarkin--Tsygan calculus of $A$.
Functoriality then yields, with almost no further work:

\begin{thmA}[= Theorem \ref{thm:invariance}]
Let $\kk$ be a field and $A$ a finite-dimensional $\kk$-algebra with
$\gldim A<\infty$.
\begin{enumerate}
\item $\sigma_A$ is an automorphism of the Tamarkin--Tsygan calculus
$\HB(A)$.
\item The assignment $m\mapsto\sigma_A^m=\HB(\ar_A^{\lten_A m})$ is a group
homomorphism $\ZZ\to\Aut_{\TT}(\HB(A))$.
\item For every two-sided tilting complex $X$ of $A$-$B$-bimodules one has
$\HB(X)\circ\sigma_A=\sigma_B\circ\HB(X)$. Consequently the isomorphism
class of the pair $(\HB(A),\sigma_A)$ is a derived invariant of $A$.
\end{enumerate}
\end{thmA}

Spelled out, part (1) says that the homological and cohomological
components $(\sigma_\bullet,\sigma^\bullet)$ of $\sigma_A$ respect all
five operations of the calculus:
\begin{gather*}
  \sigma^\bullet(\eta\cup\theta)=\sigma^\bullet\eta\cup\sigma^\bullet\theta,
  \qquad
  \sigma^\bullet[\eta,\theta]=[\sigma^\bullet\eta,\sigma^\bullet\theta],\\
  \sigma_\bullet(z\cap\eta)=\sigma_\bullet z\cap\sigma^\bullet\eta,
  \qquad
  \sigma_\bullet B=B\sigma_\bullet .
\end{gather*}

The second main result computes the homological component of $\sigma_A$
for elementary algebras and identifies it with the Coxeter transformation.

\begin{thmB}[= Theorem \ref{thm:matrix}]
Let $A$ be a finite-dimensional elementary $\kk$-algebra with
$\gldim A<\infty$, let $e_1,\dots,e_n$ be a complete set of primitive
orthogonal idempotents, and let $\bar e_1,\dots,\bar e_n$ denote their
classes in $\HH_0(A)=A/[A,A]$. Then $\HH_i(A)=0$ for $i\geq 1$, the
classes $\bar e_1,\dots,\bar e_n$ form a basis of $\HH_0(A)$, and in this
basis the matrix of $\sigma_A$ on $\HH_0(A)$ is
\[
  -\,C_A^{-1}C_A^{\mathrm{T}},
\]
where $C_A$ is the Cartan matrix of $A$.
\end{thmB}

Two consequences are drawn in the text
(Proposition~\ref{prop:naturality}, Theorem~\ref{thm:matrix}). First,
the Euler class $\ch\colon\Kn(\per A)\to\HH_0(A)$, which sends $[e_iA]$
to $\bar e_i$ and becomes an isomorphism after extension of scalars to
$\kk$, intertwines the Coxeter transformation $\Cox_A$ with $\sigma_A$.
Second,
the characteristic polynomial of $\sigma_A$ on Hochschild homology is the
Coxeter polynomial of $A$, read in $\kk[x]$:
\[
  \det\bigl(x\cdot\id-\sigma_A|_{\HH_\bullet(A)}\bigr)
  \;=\;
  \coxpol_A(x).
\]

Thus the enriched calculus $(\HB(A),\sigma_A)$ determines the
Tamarkin--Tsygan calculus (forget $\sigma$) and the Coxeter polynomial (take
the characteristic polynomial of $\sigma$ on homology, on the nose if
$\operatorname{char}\kk=0$, modulo $p$ in characteristic $p$).

The third main result identifies the cohomological component completely;
we regard it as the conceptual heart of the paper.

\begin{thmC}[= Theorem \ref{thm:identity}, Corollary \ref{cor:nakayama}]
Let $A$ be any finite-dimensional $\kk$-algebra. The left and right
actions of $\HH^\bullet(A)$ on the bimodule $\D A$ coincide.
Consequently:
\begin{enumerate}
\item if $\gldim A<\infty$, then $\sigma^\bullet_A=\id_{\HH^\bullet(A)}$,
so the enriched calculus amounts to the Tamarkin--Tsygan calculus together
with the single operator $\sigma_\bullet$ on Hochschild homology, which
commutes with the Connes differential and with every cap operator;
\item if $A$ is Frobenius, its Nakayama automorphism acts trivially on
$\HH^\bullet(A)$.
\end{enumerate}
\end{thmC}

Part (2) was proved very recently by Su\'arez-\'Alvarez
\cite{SuarezAlvarez25} by an explicit cochain-level analysis; here it
falls out of the same duality argument as part (1), the common source
being the centrality of the class of $\D A$ in the derived Picard group
of any finite-dimensional algebra (Lemma~\ref{lem:commutation} with
$B=A$). Theorem~C converts the two halves of $\sigma_A$ into a sharp
dichotomy: Hochschild cohomology, the home of deformation theory, is
blind to the Coxeter symmetry, while Hochschild homology, the home of
traces, detects it completely. The example that motivated this paper
shows that the resulting homological refinement is strict
(Example~\ref{cor:example}): for the path algebras $A$ and $B$ of the
quivers of types $\mathbb{D}_4$ and $\mathbb{A}_4$ above, over an
algebraically closed field, the Tamarkin--Tsygan calculi are isomorphic,
but there exists no isomorphism $\HB(A)\to\HB(B)$ commuting with
$\sigma_A$ and $\sigma_B$. Hence $(\HB(-),\sigma)$ is a strictly finer
derived invariant than the Tamarkin--Tsygan calculus, and also strictly
finer than the Coxeter polynomial (Remark~\ref{rem:strict}). It is not,
however, a complete derived invariant: the two smallest cospectral
trees, on eight vertices, have conjugate Coxeter matrices and hence
isomorphic enriched calculi, without being derived equivalent
(Example~\ref{ex:notcomplete}).

Finally, two classical theorems of Coxeter theory become statements
\emph{internal} to the enriched calculus. Happel's trace formula
\cite{Happel97} reads
\[
  \sum_{i\geq 0}(-1)^i\dim_\kk \HH^i(A)
  \;=\;
  -\tr\bigl(\sigma_A|_{\HH_\bullet(A)}\bigr):
\]
the Euler characteristic of the cohomological half of the calculus equals
minus the trace of the Coxeter automorphism on the homological half
(Corollary~\ref{cor:happel}); the operator $\sigma_A$ is in this precise
sense the promotion of Happel's numerical invariant to a derived-invariant
operator, and over a field of characteristic zero the full Coxeter
polynomial is recovered from the supertraces of the powers of $\sigma_A$ by
Newton's identities. And if $A$ is fractionally Calabi--Yau, in the bimodule
sense that $(\D A)^{\lten_A q}\cong A[p]$ in $\Db(A^e)$
\cite{HerschendIyama11}, then $\sigma_A^q=(-1)^{p-q}\id$ on
$\HH_\bullet(A)$, whence $\Cox_A$ is periodic and $\coxpol_A$ is a product
of cyclotomic polynomials (Corollary~\ref{cor:fcy}), recovering part of
\cite[Thm.~A]{delaPena14} at the level of the calculus.

Section~\ref{sec:identity} proves Theorem~C.
Section~\ref{sec:dg}, written following a suggestion of B.~Keller, carries
the construction to smooth and proper dg algebras, hence to perfect
derived categories of smooth projective varieties, where the enrichment
degenerates exactly on Calabi--Yau categories and where, on $\PP^n$, it is
governed by the Coxeter polynomial $(x+(-1)^n)^{n+1}$ of the Beilinson
algebra. Section~\ref{sec:discussion} relates $\sigma_A$ to recent work in
several directions: the Lefschetz- and Hirzebruch--Riemann--Roch-type
formulas of Han \cite{Han20}, which our Theorem~B categorifies in degree
zero; the categorical entropy of Serre functors
\cite{DHKK14,Han22,ChangSchroll25,KikutaOuchi} and de la Pe\~na's Mahler
measure programme \cite{delaPena15}; the periodicity dictionary between
fractional Calabi--Yau properties and cyclotomic Coxeter polynomials
\cite{Ladkani08,HerschendIyama11,CDIM25,Pfeifer25,Rognerud21}; and the
Calabi--Yau completions of Keller \cite{Keller11}, of which the operator
$\sigma_A$ is, informally, the first Fourier mode. We close with open
questions, of which we single out the $\ar$-twisted cohomological theory
of Question~\ref{q:twisted} as the natural continuation of Theorem~C.

\subsection*{Acknowledgements}
The author thanks Claude Cibils, Jos\'e Antonio de la Pe\~na and Bernhard
Keller for the conversations that shaped his view of these two subjects,
and is especially grateful to Bernhard Keller for his comments on a first
version of this paper and for suggesting the extension to smooth and
proper dg categories carried out in Section~\ref{sec:dg}. He also thanks
Ibrahim Assem for his careful reading of a first version and for remarks
that improved the exposition, in particular of the introduction and of
the motivating example.


\section{Conventions and recollections}\label{sec:conventions}

Throughout, $\kk$ is a field, $\D=\Hom_\kk(-,\kk)$ is the $\kk$-duality, and
$A$, $B$ denote finite-dimensional $\kk$-algebras. Modules are right modules
unless stated otherwise; $\rmod A$ is the category of finite-dimensional
right $A$-modules. An $A$-$B$-bimodule is a left $A$-, right $B$-module on
which $\kk$ acts centrally. We write $A^e=A\otimes_\kk A^{\mathrm{op}}$, so
that $A$-$A$-bimodules are left $A^e$-modules. We denote by $\Db(A)$ the
unbounded derived category of right $A$-modules, by $\Db^b(\rmod A)$ the
bounded derived category, and by $\per A\subseteq\Db(A)$ the perfect derived
category, that is, the thick subcategory generated by $A_A$. If
$\gldim A<\infty$, the canonical functor $\per A\to\Db^b(\rmod A)$ is an
equivalence. Since $\kk$ is a field, all derived functors between derived
categories of bimodules below are computed by means of resolutions that are
homotopically projective over the enveloping algebra, hence over each side
separately; see \cite{Keller98} for this standard framework.

\subsection{Cartan and Coxeter data}\label{subsec:coxeter}

Assume in this subsection that $A$ is \emph{elementary}, that is,
$A/\rad A\cong \kk^n$ as $\kk$-algebras; equivalently $A\cong \kk Q/I$ for a
finite quiver $Q$ and an admissible ideal $I$. Fix a complete set
$e_1,\dots,e_n$ of primitive orthogonal idempotents, and set
$P_i=e_iA$ and $I_i=\D(Ae_i)$, complete lists of the indecomposable
projective and injective right $A$-modules.

The \emph{Cartan matrix} of $A$ is $C_A=(c_{ij})\in\ZZ^{n\times n}$ with
\[
  c_{ij}\;=\;\dim_\kk\Hom_A(P_i,P_j)\;=\;\dim_\kk e_jAe_i .
\]
The \emph{dimension vector} of $M\in\rmod A$ is the column vector
$\dimv M\in\ZZ^n$ with $(\dimv M)_i=\dim_\kk Me_i=\dim_\kk\Hom_A(P_i,M)$; it
extends to $\per A$ by setting $\dimv M=\sum_j(-1)^j\dimv H^j(M)$.
Whenever $\gldim A<\infty$ it induces an isomorphism
$\dimv\colon \Kn(\per A)\to\ZZ^n$. We record the two basic computations:
\begin{equation}\label{eq:dimPI}
  \dimv P_j = C_A\,\varepsilon_j,
  \qquad
  \dimv I_j = C_A^{\mathrm{T}}\,\varepsilon_j ,
\end{equation}
where $\varepsilon_1,\dots,\varepsilon_n$ is the standard basis of
$\ZZ^n$. Indeed
\begin{gather*}
  (\dimv P_j)_i=\dim\Hom_A(P_i,P_j)=c_{ij},\\
  \Hom_A\bigl(e_iA,\D(Ae_j)\bigr)\cong\D(e_iA\otimes_AAe_j)=\D(e_iAe_j),
\end{gather*}
and the second computation gives $(\dimv I_j)_i=c_{ji}$. If $\gldim A<\infty$ then $C_A$ is invertible over
$\ZZ$ \cite{Happel88}, so \eqref{eq:dimPI} determines the class of any
perfect complex in terms of the classes $[P_i]$.

If $\gldim A<\infty$, the category $\per A=\Db^b(\rmod A)$ has a Serre
functor, namely the derived Nakayama functor $\nu=-\lta\D A$
\cite{BondalKapranov90,Happel88}, and it has Auslander--Reiten triangles
with translation $\ar=\nu\circ[-1]$, see \cite{Happel91}. The \emph{Coxeter
transformation} of $A$ is
\[
  \Cox_A:=\Kn(\ar)\colon\Kn(\per A)\longrightarrow\Kn(\per A),
\]
and the \emph{Coxeter polynomial} is
$\coxpol_A(x)=\det(x\cdot\id-\Cox_A)\in\ZZ[x]$.

\begin{lemma}\label{lem:coxmatrix}
In the basis $[P_1],\dots,[P_n]$ of $\Kn(\per A)$ the matrix of $\Cox_A$ is
$-C_A^{-1}C_A^{\mathrm{T}}$. Consequently
$\coxpol_A(x)=\det\!\bigl(xI_n+C_A^{-1}C_A^{\mathrm{T}}\bigr)$, and this
polynomial is unchanged if $C_A$ is replaced by $C_A^{\mathrm{T}}$ or by
$PC_AP^{\mathrm{T}}$ for any $P\in GL_n(\ZZ)$.
\end{lemma}

\begin{proof}
There is an isomorphism of right $A$-modules $P_i\otimes_A\D A=e_i\D A\cong
\D(Ae_i)=I_i$, given by restriction of linear forms; hence
$\nu(P_i)\cong I_i$ and $\Cox_A[P_i]=[\,\ar P_i\,]=-[I_i]$. Writing
$[I_i]=\sum_ja_j[P_j]$ and applying $\dimv$ together with
\eqref{eq:dimPI} gives $C_Aa=C_A^{\mathrm{T}}\varepsilon_i$, that is,
$a=C_A^{-1}C_A^{\mathrm{T}}\varepsilon_i$. For the last assertions, note
$\det C_A=\pm 1$ and
$\det(xI+C^{-1}C^{\mathrm{T}})=\det(C)^{-1}\det(xC+C^{\mathrm{T}})$; the
expression $\det(xC+C^{\mathrm{T}})$ is invariant under
$C\mapsto PCP^{\mathrm{T}}$ since $\det(P)^2=1$, and it is invariant under
$C\mapsto C^{\mathrm{T}}$ because
$\det(xC^{\mathrm{T}}+C)=\det\bigl((xC^{\mathrm{T}}+C)^{\mathrm{T}}\bigr)
=\det(xC+C^{\mathrm{T}})$; compare
\cite[\S2]{Ladkani21}.
\end{proof}

It is classical that $\coxpol_A$ is a derived invariant; see
\cite[Prop.~2.3]{delaPena14} and \cite{LenzingdelaPena08}; for its
behaviour under one-point extensions see \cite{Happel09}. We shall recover
the derived invariance in Corollary~\ref{cor:coxinv} below.

\subsection{The Tamarkin--Tsygan calculus and the functor
$\HB$}\label{subsec:calculus}

A \emph{Tamarkin--Tsygan calculus}
$(\mathcal H_\bullet,\mathcal H^\bullet,\cup,[-,-],\cap,B)$ consists of a
Gerstenhaber algebra $(\mathcal H^\bullet,\cup,[-,-])$, a graded module
structure $\cap\colon\mathcal H_n\otimes\mathcal H^m\to\mathcal H_{n-m}$
over it, and a differential $B\colon\mathcal H_\bullet\to
\mathcal H_{\bullet+1}$ with $B^2=0$, subject to the compatibilities of
noncommutative differential calculus, in particular the Cartan homotopy
formula relating $B$, $\cap$ and the Lie action; see
\cite{DaletskiiGelfandTsygan90,TamarkinTsygan00} and
\cite{Armenta19} for the precise axioms used here. A \emph{morphism of
calculi} $(f,g)\colon H\to K$ is a pair of graded linear maps
$f\colon\mathcal H_\bullet\to\mathcal K_\bullet$,
$g\colon\mathcal H^\bullet\to\mathcal K^\bullet$ commuting with $\cup$,
$[-,-]$, $\cap$ and $B$ in the obvious sense \cite{Armenta19}. We write
$\TT$ for the resulting category.

The Hochschild theory of $A$ defines an object
\[
  \HB(A)=\bigl(\HH_\bullet(A),\HH^\bullet(A),\cup_A,[-,-]_A,\cap_A,B_A
  \bigr)\in\TT .
\]
Following \cite{Armenta19}, let $\Ak$ be the category whose objects are the
finite-dimensional $\kk$-algebras and whose morphisms $X\colon A\to B$ are
the complexes of $A$-$B$-bimodules that are isomorphic, in the derived
category of $A$-$B$-bimodules, to a bounded complex whose components are
finitely generated projective as right $B$-modules; composition is the
derived tensor product, $Y\circ X = X\ltb Y$ for $X\colon A\to B$ and
$Y\colon B\to C$, and the identity of $A$ is the bimodule $A$. Two
morphisms that are isomorphic in the derived category of bimodules are
equal in $\Ak$. The main theorem of \cite{Armenta19} (see also
\cite{ArmentaKeller17,ArmentaKeller19}) is:

\begin{theorem}[{\cite{Armenta19}}]\label{thm:thesis}
The assignments $A\mapsto\HB(A)$ and $X\mapsto\HB(X)$, with components
$\HB(X)=(\HH_\bullet(X),\HH^\bullet(X))$, define a functor
\[
  \HB\colon\Ak\longrightarrow\TT
\]
which sends two-sided tilting complexes to isomorphisms and is constant on
each class of derived equivalent algebras. The homological components
$\HH_\bullet(X)$ coincide with the morphisms induced by Keller's cyclic
functor \cite{Keller98}; the cohomological components $\HH^\bullet(X)$ of an
invertible $X$ respect the cup product \cite{Rickard89a}, the Gerstenhaber
bracket \cite{Keller04}, and together with $\HH_\bullet(X)$ the cap product
\cite{ArmentaKeller17,Armenta19} and the Connes differential
\cite{ArmentaKeller19}.
\end{theorem}

\subsection{Two-sided tilting complexes}\label{subsec:tilting}

A morphism $X\colon A\to B$ of $\Ak$ is a \emph{two-sided tilting complex}
if it is invertible, that is, if there exists a complex $Y$ of
$B$-$A$-bimodules with $X\ltb Y\cong A$ in $\Db(A^e)$ and
$Y\lta X\cong B$ in $\Db(B^e)$. By Rickard's theorems
\cite{Rickard89,Rickard91}, $A$ and $B$ are derived equivalent if and only
if a two-sided tilting complex $X\colon A\to B$ exists; moreover $X$ is then
perfect as a complex of left $A$-modules and as a complex of right
$B$-modules, and the inverse is computed by either one-sided dual:
\begin{equation}\label{eq:duals}
  Y\;\cong\;\RHom_B(X_B,B_B)\;\cong\;\RHom_{A}(\,{}_AX,\,{}_AA)
\end{equation}
as complexes of $B$-$A$-bimodules, see
\cite[\S4]{Rickard91} and \cite[Ch.~6]{Zimmermann14}. Derived equivalences
preserve finiteness of global dimension \cite{Happel88}.

We will use repeatedly the following elementary representability statement.

\begin{lemma}\label{lem:representative}
Let $X$ be a complex of $A$-$B$-bimodules with bounded, finite-dimensional
total cohomology, which is perfect as a complex of right $B$-modules. Then
$X$ is isomorphic, in the derived category of $A$-$B$-bimodules, to a
bounded complex whose components are finitely generated projective right
$B$-modules; in particular $X\in\Ak(A,B)$.
\end{lemma}

\begin{proof}
Replacing $X$ by the total complex of a bimodule bar resolution
$\cdots\to X\otimes_\kk B\otimes_\kk B\to X\otimes_\kk B\to X$, we may
assume the components of $X$ are of the form $V\otimes_\kk B$ with $V$
finite-dimensional, hence finitely generated free as right $B$-modules,
at the price of unboundedness to the left. Brutally truncate: for
$N\gg 0$, the kernel $K$ of the differential out of degree $-N$ is a
sub-bimodule, finitely generated over $B$, and by the syzygy argument (the
truncated complex resolves an object that is perfect over $B$) $K$ is
projective as a right $B$-module once $N$ exceeds the relevant projective
dimension. The bounded complex $K\to X^{-N+1}\to\cdots$ represents $X$.
\end{proof}

\section{The Auslander--Reiten bimodule}\label{sec:ARbimodule}

\begin{lemma}\label{lem:invertible}
Let $A$ be a finite-dimensional $\kk$-algebra with $\gldim A<\infty$. Then
the bimodule $\D A$ is a two-sided tilting complex over $A$, with inverse
$\RHom_A(\D A,A)$, and $\D A$, $\D A[-1]$ and their inverses are morphisms
$A\to A$ in $\Ak$.
\end{lemma}

\begin{proof}
Since $\gldim A<\infty$ and $\dim_\kk A<\infty$, the right $A$-module
$\D A$ and the left $A$-module $\D A$ are perfect; by
Lemma~\ref{lem:representative}, $\D A$ and $\D A[-1]$ are morphisms in
$\Ak$. Set $Y=\RHom_A(\D A,A)$, the derived $\Hom$ of right $A$-modules,
an object of the derived category of $A$-$A$-bimodules via the left
structures of $\D A$ and of $A$.

We prove that $Y$ is a two-sided inverse of $\D A$ in the monoidal
category $(\Db(A^e),\lta,A)$. All derived functors are computed with
h-projective ($P$ is h-projective when $Hom_A^\bullet (P,N)$ is acyclic for all acyclic complex $N$) resolutions of complexes of bimodules; since $\kk$ is a
field, these restrict to h-projective one-sided modules, so the canonical
morphisms below are defined at the level of complexes of bimodules
\cite{Keller98}.

\emph{Step 1.} For $M\in\per A$ and a complex $N$ of $A$-bimodules,
consider the canonical morphism
\[
  N\lta\RHom_A(M,A)\longrightarrow\RHom_A(M,N),\qquad
  n\otimes f\longmapsto\bigl(m\mapsto n\,f(m)\bigr).
\]
It is balanced over $A$, because the left action on $\RHom_A(M,A)$ is
$(af)(m)=a\,f(m)$, and it is a morphism of complexes of $A$-bimodules for
the remaining structures (the left structure of $N$, and the right action
$(fa)(m)=f(am)$ coming from the left structure of $M$). It is invertible
for $M=A$, both sides are triangulated functors of $M$, and $\per A$ is
generated by $A$; hence it is invertible for every $M\in\per A$. Taking
$M=N=\D A$, which is legitimate since $\D A\in\per A$, yields
\[
  \D A\lta Y\;\cong\;\RHom_A(\D A,\D A)
  \qquad\text{in }\Db(A^e).
\]
\emph{Step 2.} The tensor--hom adjunction over $\kk$ gives a natural
isomorphism $\RHom_A\bigl(X,\D V\bigr)\cong\D\bigl(X\lta V\bigr)$ for $X$ a
complex of right $A$-modules and $V$ a complex of left $A$-modules,
compatible with all auxiliary bimodule structures. With $X=\D A$ and
$V=A$ (so $\D V=\D A$):
\[
  \RHom_A(\D A,\D A)\;\cong\;\D\bigl(\D A\lta A\bigr)\;=\;\D(\D A)\;\cong\;
  A\qquad\text{in }\Db(A^e),
\]
the last isomorphism because $A$ is finite-dimensional. Hence
$\D A\lta Y\cong A$.

\emph{Step 3.} The mirror of Steps 1--2 over $A^{\mathrm{op}}$ --- using
the balanced morphism
$\RHom_{A^{\mathrm{op}}}(M,A)\lta N\to\RHom_{A^{\mathrm{op}}}(M,N)$,
$f\otimes n\mapsto(m\mapsto f(m)\,n)$, for $M\in\per A^{\mathrm{op}}$,
together with the left-handed tensor--hom adjunction
$\RHom_{A^{\mathrm{op}}}(W,\D V)\cong\D(V\lta W)$ --- yields, for
$Y':=\RHom_{A^{\mathrm{op}}}(\D A,A)$,
\[
  Y'\lta\D A\;\cong\;\RHom_{A^{\mathrm{op}}}(\D A,\D A)\;\cong\;
  \D(A\lta\D A)\;\cong\;A.
\]
By associativity and unitality of $\lta$,
\[
  Y\;\cong\;A\lta Y\;\cong\;(Y'\lta\D A)\lta Y\;\cong\;
  Y'\lta(\D A\lta Y)\;\cong\;Y'\lta A\;\cong\;Y',
\]
whence $Y\lta\D A\cong Y'\lta\D A\cong A$. Thus $\D A$ is a two-sided
tilting complex with inverse $Y\cong Y'$, and then so are its shifts, with
$(\D A[-1])^{-1}=Y[1]$.
\end{proof}

\begin{remark}
The lemma is classical: for a Gorenstein algebra --- in particular one of
finite global dimension --- the bimodule $\D A$ generates a copy of $\ZZ$
inside the derived Picard group $\mathrm{DPic}(A)$, and the associated
standard self-equivalence of $\per A$ is the derived Nakayama functor
$\nu=-\lta\D A$ of Happel \cite{Happel88,Happel91}; see
\cite{Rickard91,RouquierZimmermann03,Yekutieli99,Zimmermann14}. We have
included the short proof in order to keep the bimodule conventions of
\cite{Armenta19} explicit.
\end{remark}

\begin{definition}\label{def:ARbimodule}
Let $\gldim A<\infty$. The \emph{Serre bimodule} of $A$ is
$\serre_A:=\D A$, and the \emph{Auslander--Reiten bimodule} of $A$ is
\[
  \ar_A:=\D A[-1],
\]
both viewed as (invertible) morphisms $A\to A$ in $\Ak$. The associated
standard equivalences of $\per A$ are the Serre functor
$\nu=-\lta\serre_A$ and the derived Auslander--Reiten translation
$\ar=-\lta\ar_A$ \cite{Happel91}, and $\Kn(-\lta\ar_A)=\Cox_A$ by
Lemma~\ref{lem:coxmatrix}.
\end{definition}

The key property of $\ar_A$ is that it commutes with every derived
equivalence. This is the bimodule-level form of the uniqueness of Serre
functors.

\begin{lemma}\label{lem:commutation}
Let $A$, $B$ be finite-dimensional algebras of finite global dimension and
let $X\colon A\to B$ be a two-sided tilting complex. Then there are
isomorphisms in the derived category of $A$-$B$-bimodules
\[
  \serre_A\lta X\;\cong\;X\ltb\serre_B,
  \qquad
  \ar_A\lta X\;\cong\;X\ltb\ar_B .
\]
\end{lemma}

\begin{proof}
The second isomorphism follows from the first by shifting. For the first,
we exhibit both sides as the $\kk$-dual of the inverse of $X$. Choose by
Lemma~\ref{lem:representative} a representative of $X$ that is bounded with
components finitely generated projective over $B$ on the right; since $X$ is
also perfect over $A$ on the left \cite{Rickard91}, fix similarly a
representative adapted to the left when needed.

\emph{(a)} The map
\[
  \alpha_X\colon\; X\ltb \D B\longrightarrow
  \D\bigl(\RHom_B(X,B)\bigr),
  \qquad
  x\otimes\varphi\longmapsto\bigl(f\mapsto\varphi(f(x))\bigr),
\]
is well defined: it is balanced over $B$ because
$\varphi(f(xb))=\varphi(f(x)b)=(b\cdot\varphi)(f(x))$, and it is a morphism
of complexes of $A$-$B$-bimodules for the standard structures. It is an
isomorphism when $X=B$, both sides are triangulated functors of the
underlying complex of right $B$-modules, and quasi-isomorphisms are
detected after forgetting the left $A$-structure; hence $\alpha_X$ is an
isomorphism for $X$ perfect over $B$.

\emph{(b)} Symmetrically, the map
\[
  \beta_X\colon\;\D A\lta X\longrightarrow
  \D\bigl(\RHom_{A}({}_AX,{}_AA)\bigr),
  \qquad
  \varphi\otimes x\longmapsto\bigl(f\mapsto\varphi(f(x))\bigr),
\]
is balanced over $A$ because $f$ is left $A$-linear,
$\varphi a\otimes x$ and $\varphi\otimes ax$ having equal images, and is an
isomorphism of complexes of $A$-$B$-bimodules for $X$ perfect over $A$ on
the left, by the same d\'evissage applied on the left.

\emph{(c)} By \eqref{eq:duals} the complexes $\RHom_B(X,B)$ and
$\RHom_A({}_AX,{}_AA)$ are isomorphic as complexes of $B$-$A$-bimodules,
both computing the inverse $X^{-1}$. Applying $\D$ and combining with (a)
and (b):
\[
  \serre_A\lta X\;\xrightarrow{\;\beta_X\;}\;\D\bigl(X^{-1}\bigr)
  \;\xleftarrow{\;\alpha_X\;}\;X\ltb\serre_B . \qedhere
\]
\end{proof}

\section{The Coxeter automorphism of the calculus}\label{sec:automorphism}

\begin{definition}\label{def:sigma}
Let $A$ be a finite-dimensional $\kk$-algebra with $\gldim A<\infty$. The
\emph{Coxeter automorphism} of the Tamarkin--Tsygan calculus of $A$ is
\[
  \sigma_A\;:=\;\HB(\ar_A)\;=\;\HB\bigl(\D A[-1]\bigr)
  \;\colon\;\HB(A)\longrightarrow\HB(A).
\]
We write $\sigma_\bullet$ and $\sigma^\bullet$ for its homological and
cohomological components.
\end{definition}

\begin{theorem}\label{thm:invariance}
Let $\kk$ be a field and $A$ a finite-dimensional $\kk$-algebra with
$\gldim A<\infty$.
\begin{enumerate}
\item $\sigma_A$ is an automorphism of $\HB(A)$ in $\TT$, with inverse
$\HB(\ar_A^{-1})$. In particular $\sigma^\bullet$ is an automorphism of the
Gerstenhaber algebra $\HH^\bullet(A)$, the pair
$(\sigma_\bullet,\sigma^\bullet)$ satisfies the projection formula
$\sigma_\bullet(z\cap\eta)=\sigma_\bullet z\cap\sigma^\bullet\eta$, and
$\sigma_\bullet$ commutes with the Connes differential $B$.
\item The map $\ZZ\to\Aut_{\TT}(\HB(A))$, $m\mapsto
\sigma_A^{\,m}=\HB\bigl(\ar_A^{\lten_A m}\bigr)$, is a group homomorphism.
\item For every two-sided tilting complex $X\colon A\to B$ one has
$\HB(X)\circ\sigma_A=\sigma_B\circ\HB(X)$ in $\TT$. Consequently the
isomorphism class of the pair $(\HB(A),\sigma_A)$ in the category of
Tamarkin--Tsygan calculi equipped with an automorphism is invariant under
derived equivalence.
\end{enumerate}
\end{theorem}

\begin{proof}
(1) Lemma~\ref{lem:invertible} shows that $\ar_A$ is invertible in $\Ak$,
with inverse given by $\RHom_A(\D A,A)[1]$, and by
Theorem~\ref{thm:thesis} the functor $\HB$ sends the relations
$\ar_A^{-1}\circ\ar_A=\id_A$ and $\ar_A\circ\ar_A^{-1}=\id_A$ of $\Ak$ to
the corresponding relations in $\TT$. Every morphism of $\TT$ commutes by
definition with $\cup$, $[-,-]$, $\cap$ and $B$; an isomorphism whose
cohomological component preserves $\cup$ and $[-,-]$ is an automorphism of
the Gerstenhaber algebra.

(2) Functoriality of $\HB$ and associativity of $\lten$ in $\Ak$.

(3) In $\Ak$ the two composites $A\to B$ in question are represented by the
bimodule complexes $X\ltb\ar_B$ and $\ar_A\lta X$, which are isomorphic by
Lemma~\ref{lem:commutation}; since morphisms of $\Ak$ are taken up to
isomorphism in the derived category of bimodules, the two composites are
\emph{equal} in $\Ak$, and we apply $\HB$:
\[
  \sigma_B\circ\HB(X)=\HB(X\ltb\ar_B)=\HB(\ar_A\lta X)
  =\HB(X)\circ\sigma_A .
\]
If $F\colon\Db(A)\to\Db(B)$ is any equivalence, by Rickard's theorem we may
choose a two-sided tilting complex $X$ \cite{Rickard91}; then $\HB(X)$ is an
isomorphism $\HB(A)\to\HB(B)$ in $\TT$ commuting with the $\sigma$'s, which
is the claimed invariance.
\end{proof}

\begin{remark}\label{rem:dpic}
Theorem~\ref{thm:thesis} restricts to an action of the derived Picard group
$\mathrm{DPic}(A)$ of \cite{Yekutieli99,RouquierZimmermann03} on the object
$\HB(A)$, i.e.\ a group homomorphism from $\mathrm{DPic}(A)$ to
$\Aut_{\TT}(\HB(A))$. The automorphism $\sigma_A$ is the
image of the class of $\ar_A$, which by Lemma~\ref{lem:commutation} is
fixed by the isomorphisms $\mathrm{DPic}(A)\cong\mathrm{DPic}(B)$ induced by
derived equivalences. Theorem~\ref{thm:invariance}(3) is the shadow of this
naturality.
\end{remark}

\begin{remark}\label{rem:models}
The thesis \cite{Armenta19} constructs three isomorphic models
$\HB$, $\widehat{\HB}$, $\widetilde{\HB}$ of the calculus and compatible
comparison isomorphisms. All statements about $(\HB(A),\sigma_A)$ in this
paper are invariant under these comparisons.
\end{remark}

\begin{remark}
No elementarity hypothesis enters Theorem~\ref{thm:invariance}: it holds for
every finite-dimensional algebra of finite global dimension over an
arbitrary field.
\end{remark}

\section{Hochschild homology in degree zero and the Coxeter
polynomial}\label{sec:computation}

In this section $A$ is a finite-dimensional \emph{elementary}
$\kk$-algebra of finite global dimension with $n$ simple modules; we keep
the notation of
\S\ref{subsec:coxeter}.

\subsection{Concentration in degree zero}

\begin{proposition}\label{prop:concentration}
$\HH_i(A)=0$ for all $i\geq 1$, and the classes
$\bar e_1,\dots,\bar e_n$ form a $\kk$-basis of $\HH_0(A)=A/[A,A]$.
\end{proposition}

\begin{proof}
Since $A/\rad A\cong\kk^n$ is commutative, $[A,A]\subseteq\rad A$;
since $\gldim A<\infty$, Lenzing's theorem gives
$\rad A\subseteq[A,A]$, see \cite[\S5]{Lenzing69}. Hence
$A/[A,A]=A/\rad A=\kk^n$ with basis the $\bar e_i$. The vanishing of the
higher Hochschild homology is \cite[Prop.~2.5]{Keller98}; see also
\cite[Prop.~1]{Han20} for this formulation.
\end{proof}

\subsection{Euler classes}

For a finitely generated projective right $A$-module $P$, written
$P=\operatorname{im}(e)$ for an idempotent matrix $e=(e_{st})\in M_r(A)$,
the \emph{Hattori--Stallings trace} of an endomorphism $f$ of $P$,
represented by a matrix $(f_{st})\in M_r(A)$ with $f=efe$, is
\[
  \tr_P(f)\;:=\;\sum_{s}\overline{f_{ss}}\ \in\ \HH_0(A)=A/[A,A];
\]
it is independent of all choices and additive \cite{Hattori65,Stallings65}.
The \emph{Euler class} of $M\in\per A$ is
\[
  \ch(M)\;:=\;\sum_j(-1)^j\,\tr_{P^j}(\id)\ \in\ \HH_0(A)
\]
for any bounded complex $P^\bullet$ of finitely generated projectives
isomorphic to $M$ in $\Db(A)$; equivalently $\ch$ is the composition of the
canonical isomorphism between $\Kn(\per A)$ and
$\Kn(\operatorname{proj}A)$ with the trace of the identity. It is additive on triangles, satisfies
$\ch(M[1])=-\ch(M)$, and $\ch(P_i)=\bar e_i$.

\begin{lemma}\label{lem:chformula}
For every $M\in\per A$ one has
$\ch(M)=\sum_i a_i\,\bar e_i$ with
$a=C_A^{-1}\,\dimv M\in\ZZ^n$. In particular
$\ch\otimes\kk\colon\Kn(\per A)\otimes_\ZZ\kk\to\HH_0(A)$ is an
isomorphism sending $[P_i]\otimes 1$ to $\bar e_i$.
\end{lemma}

\begin{proof}
Both $\ch$ and $a(M):=C_A^{-1}\dimv M$ are additive on triangles, and they
agree on the generators $P_i$ of $\Kn(\per A)$ by \eqref{eq:dimPI}; the
classes $[P_i]$ form a $\ZZ$-basis of $\Kn(\per A)$.
\end{proof}

\subsection{Naturality of Euler classes}

\begin{proposition}\label{prop:naturality}
Let $X\colon A\to B$ be a morphism of $\Ak$ between elementary algebras of
finite global dimension, represented by a bounded complex of
$A$-$B$-bimodules whose components $X^j$ are finitely generated projective
right $B$-modules. Then:
\begin{enumerate}
\item the degree-zero homological component of $\HB(X)$ is the
\emph{tensor-trace} map
\[
  \HH_0(X)\colon A/[A,A]\longrightarrow B/[B,B],
  \qquad
  \bar a\longmapsto\sum_j(-1)^j\,\tr_{X^j}(\lambda_a),
\]
where $\lambda_a$ denotes left multiplication by $a$ on $X^j$;
\item for all $M\in\per A$,
\[
  \HB(X)_0\bigl(\ch(M)\bigr)\;=\;\ch\bigl(M\lta X\bigr).
\]
\end{enumerate}
\end{proposition}

\begin{proof}
(1) By Theorem~\ref{thm:thesis} the homological components of $\HB(X)$
coincide with the morphisms induced by Keller's cyclic functor
\cite{Keller98}, which in weight zero is the classical map induced on
Hochschild complexes by the bimodule $X$: the composite of the inclusion of
$A$ into the endomorphism dg algebra of $X$ over $B$ with the generalized
trace. Its zeroth homology is the displayed formula; see
\cite[\S3.3]{Han20}, and \cite[\S2]{Keller98}, \cite{Shklyarov13} for the
same map in the dg framework.

(2) Both sides are additive in $[M]\in\Kn(\per A)$ --- the left side
because $\ch$ is and $\HB(X)_0$ is linear; the right side because
$-\lta X$ is triangulated and $\ch_B$ is additive on triangles --- so it
suffices to treat $M=P_i=e_iA$. Then $M\lta X\cong e_iX$, a bounded
complex with components $e_iX^j$ finitely generated projective over $B$,
and
\begin{align*}
  \ch(e_iX)&=\sum_j(-1)^j\tr_{e_iX^j}(\id)
            =\sum_j(-1)^j\tr_{X^j}(\lambda_{e_i})\\
           &=\HH_0(X)(\bar e_i)=\HB(X)_0\bigl(\ch(P_i)\bigr),
\end{align*}
where the middle equality is the standard property of the Hattori--Stallings
trace: for an idempotent endomorphism $\pi$ of a finitely generated
projective module $Q$, $\tr_{\operatorname{im}\pi}(\id)=\tr_Q(\pi)$.
\end{proof}

\begin{remark}\label{rem:verify}
Part (1) is the single point of this paper where the normalization of the
maps $\HH_\bullet(X)$ of \cite{Armenta19} is compared with the classical
tensor-trace; the comparison is the content of the main theorem of
\cite{Armenta19} --- the identification of $\HH_\bullet(X)$ with the map
induced by Keller's cyclic functor --- together with the construction of
\cite[\S2]{Keller98}, and is
independent of the choice among the three models of
Remark~\ref{rem:models}.
\end{remark}

\subsection{The matrix of the Coxeter automorphism}

\begin{theorem}\label{thm:matrix}
Let $A$ be a finite-dimensional elementary $\kk$-algebra with
$\gldim A<\infty$. In the basis $\bar e_1,\dots,\bar e_n$ of
$\HH_0(A)=\HH_\bullet(A)$ the matrix of the homological component
$\sigma_\bullet$ of the Coxeter automorphism is
\[
  -\,C_A^{-1}C_A^{\mathrm{T}} .
\]
The Euler class intertwines the Coxeter transformation with
$\sigma_\bullet$, that is, the square
\[
  \begin{array}{ccc}
   \Kn(\per A)\otimes_\ZZ\kk & \xrightarrow{\;\Cox_A\otimes\kk\;} &
   \Kn(\per A)\otimes_\ZZ\kk\\[3pt]
   \ch\otimes\kk\Big\downarrow & & \Big\downarrow\ch\otimes\kk\\[3pt]
   \HH_0(A) & \xrightarrow{\;\;\sigma_\bullet\;\;} & \HH_0(A)
  \end{array}
\]
commutes, and
\[
  \det\bigl(x\cdot\id-\sigma_\bullet\bigr)\;=\;
  \det\bigl(xI_n+C_A^{-1}C_A^{\mathrm{T}}\bigr)\;=\;\coxpol_A(x)
  \quad\text{in }\kk[x].
\]
\end{theorem}

\begin{proof}
By Proposition~\ref{prop:concentration}, $\HH_\bullet(A)=\HH_0(A)$ with
basis the $\bar e_i=\ch(P_i)$. By Proposition~\ref{prop:naturality}(2)
applied to the morphism $\ar_A=\D A[-1]$ of $\Ak$ (a bounded
right-projective representative exists by
Lemmas~\ref{lem:representative} and~\ref{lem:invertible}),
\[
  \sigma_\bullet(\bar e_i)
  =\HB(\ar_A)_0\bigl(\ch(P_i)\bigr)
  =\ch\bigl(P_i\lta\D A[-1]\bigr)
  =-\,\ch\bigl(e_i\D A\bigr)
  =-\,\ch(I_i),
\]
using the isomorphism $e_i\D A\cong\D(Ae_i)=I_i$ of right $A$-modules from
the proof of Lemma~\ref{lem:coxmatrix}. By Lemma~\ref{lem:chformula} and
\eqref{eq:dimPI},
$\ch(I_i)=\sum_j(C_A^{-1}C_A^{\mathrm{T}})_{ji}\,\bar e_j$, which is the
asserted matrix. The commuting square restates this computation through
Lemma~\ref{lem:coxmatrix}, since $\Cox_A$ has the same matrix
$-C_A^{-1}C_A^{\mathrm{T}}$ in the basis $[P_i]$ and
$\ch\otimes\kk$ sends $[P_i]$ to $\bar e_i$. The characteristic polynomial
is then $\coxpol_A$ by Lemma~\ref{lem:coxmatrix}, read in $\kk[x]$.
\end{proof}

\begin{corollary}\label{cor:coxinv}
The Coxeter polynomial of a finite-dimensional elementary algebra of finite
global dimension is a derived invariant.
\end{corollary}

\begin{proof}
If $A$ and $B$ are derived equivalent, Theorem~\ref{thm:invariance}(3)
provides an isomorphism of pairs
$(\HH_\bullet(A),\sigma_\bullet)\cong(\HH_\bullet(B),\sigma_\bullet)$, so
the characteristic polynomials agree in $\kk[x]$; over a field of
characteristic zero this already returns equality in $\ZZ[x]$ by
integrality of both polynomials. In general, the matrices
$-C^{-1}C^{\mathrm{T}}$ are even conjugate over $\ZZ$, since the
isomorphism of Theorem~\ref{thm:invariance}(3) restricts, by
Proposition~\ref{prop:naturality}(2), to an isomorphism of the lattices
spanned by the $\bar e_i$ intertwining the $\sigma_\bullet$'s. This
recovers the classical statement of \cite[Prop.~2.3]{delaPena14},
\cite{LenzingdelaPena08}.
\end{proof}

\subsection{The motivating example, revisited}

\begin{example}\label{cor:example}
Let $\kk$ be an algebraically closed field, $A=\kk Q_A$ the path algebra of
the $\mathbb D_4$-quiver with three arrows into the central vertex, and
$B=\kk Q_B$ the path algebra of the linearly oriented $\mathbb A_4$-quiver,
as in \S\ref{sec:intro}. Then:
\begin{enumerate}
\item $\HB(A)\cong\HB(B)$ in $\TT$;
\item there exists no isomorphism $\HB(A)\to\HB(B)$ in $\TT$ commuting
with $\sigma_A$ and $\sigma_B$; in particular $(\HB(A),\sigma_A)\not\cong
(\HB(B),\sigma_B)$ although the underlying calculi are isomorphic.
\end{enumerate}
\end{example}

\begin{proof}
(1) Both quivers are trees, so by \cite{Cibils86,Cibils90} the Hochschild
homology and cohomology of $A$ and $B$ vanish in positive degrees, with
$\HH_0\cong\kk^4$ and $\HH^0\cong\kk$; all components of the calculus other
than the unit cup product and the degree-$(0,0)$ cap product vanish, and
the two calculi are isomorphic; this is the closing remark of
\cite{Armenta19}.

(2) Ordering the vertices as in \S\ref{sec:intro}, the Cartan matrices and
the matrices $S=-C^{-1}C^{\mathrm{T}}$ of the $\sigma_\bullet$'s are
\[
  C_A=\begin{pmatrix}1&0&0&0\\1&1&0&0\\1&0&1&0\\1&0&0&1\end{pmatrix},
  \qquad
  C_B=\begin{pmatrix}1&1&1&1\\0&1&1&1\\0&0&1&1\\0&0&0&1\end{pmatrix},
\]
with characteristic polynomials
\begin{align*}
  \det(x\cdot\id-\sigma_{A,\bullet})&=\coxpol_A(x)=x^4+x^3+x+1,\\
  \det(x\cdot\id-\sigma_{B,\bullet})&=\coxpol_B(x)=x^4+x^3+x^2+x+1
\end{align*}
by Theorem~\ref{thm:matrix}. An isomorphism of pairs would conjugate
$\sigma_{A,\bullet}$ into $\sigma_{B,\bullet}$ and force the equality of
their characteristic polynomials in $\kk[x]$; but
$\coxpol_B-\coxpol_A=x^2\neq 0$ in $\kk[x]$ for every field $\kk$.
\end{proof}

\begin{remark}\label{rem:strict}
Example~\ref{cor:example} shows that the enriched calculus
$(\HB(-),\sigma)$ is a strictly finer derived invariant than the
Tamarkin--Tsygan calculus $\HB(-)$. It is strictly finer than the Coxeter
polynomial as well: it determines $\coxpol$ by Theorem~\ref{thm:matrix},
while already the underlying calculus separates many algebras with equal
Coxeter polynomial --- by Happel's formula such algebras have equal Euler characteristic of Hochschild cohomology,
but any two of them with non-isomorphic graded algebras $\HH^\bullet$ are
separated by $\HB$; see the discussions in \cite{LenzingdelaPena08}.
\end{remark}

\begin{remark}\label{rem:classical}
The derived inequivalence of $\kk Q_{\mathbb A_4}$ and
$\kk Q_{\mathbb D_4}$ is itself classical and needs none of this: for a
Dynkin quiver the bounded derived category is described completely by
Happel \cite{Happel88,Happel91}, with Auslander--Reiten quiver
$\ZZ\Delta$, and two Dynkin path algebras are derived equivalent if and
only if their underlying diagrams agree --- here the numbers of
indecomposable objects, $10$ and $12$, already differ. The content of
Example~\ref{cor:example} is not the inequivalence of the two algebras
but the comparison of two invariants: the Tamarkin--Tsygan calculus does
not separate them, the enriched calculus does, and the gap between the
two is measured exactly by the Coxeter polynomial.
\end{remark}

\begin{remark}
The polynomials above factor as
$\coxpol_A=\Phi_2(x)^2\,\Phi_6(x)$ and $\coxpol_B=\Phi_5(x)$ in terms of
cyclotomic polynomials, reflecting the exponents $1,3,3,5$ and Coxeter
number $6$ of $\mathbb D_4$, and $1,2,3,4$ and Coxeter number $5$ of
$\mathbb A_4$. Amusingly, the same pair of algebras appears in
\cite[Ex.~3.4]{Ladkani21} as the minimal example showing that the Coxeter
polynomial of an insertion is not determined by the Coxeter polynomial of
the inserted poset.
\end{remark}

\begin{remark}[Beyond global dimension one]\label{rem:gldim}
A natural question is what happens in global dimension $\geq2$.
Theorems~\ref{thm:invariance} and~\ref{thm:identity} require only
$\gldim A<\infty$, and Theorem~\ref{thm:matrix} elementarity in addition;
nothing in the theory is special to the hereditary case. Moreover,
genuinely new spectra appear in higher global dimension, and the enriched
calculus detects this internally. By Happel's vanishing theorem
\cite{Happel89hh}, a connected hereditary algebra has $\HH^{i}=0$ for
$i\geq2$ and $\HH^0=\kk$, so Corollary~\ref{cor:happel} forces
\[
  -\tr\bigl(\sigma_\bullet\bigr)
  \;=\;\sum_{i\geq0}(-1)^i\dim_\kk\HH^i(A)
  \;=\;1-\dim_\kk\HH^1(A)\;\leq\;1
\]
for every connected algebra derived equivalent to a hereditary one. The
Beilinson algebra $B_2$ of the projective plane
(Example~\ref{ex:projspace}), which has global dimension two, has Coxeter
polynomial $(x+1)^3$ and hence $\tr(\sigma_\bullet)=-3$: the inequality
fails, so $B_2$ is not derived equivalent to any hereditary algebra, and
the enriched calculus itself witnesses the obstruction.
\end{remark}

\subsection{The enriched calculus is not a complete invariant}
\label{subsec:notcomplete}
The refinement of Example~\ref{cor:example} has a precise limit, which we
now exhibit. For path algebras of trees the entire calculus degenerates,
so the enriched calculus retains exactly one piece of information: the
conjugacy class of the Coxeter matrix.

\begin{example}\label{ex:notcomplete}
Let $\kk$ be a field of characteristic zero, let $T$ be the spider with
one arm of length three and four arms of length one, and let $T'$ be the
double star consisting of two adjacent vertices carrying three leaves
each --- the two trees on eight vertices pictured below, taken with any
orientations:
\[
  Q_T:\quad
  \begin{array}{ccccc}
  2 &  & 3 &  & \\[-2pt]
    & \nwarrow & \uparrow &  & \\[-2pt]
   &  & 1 & \longrightarrow & 6\longrightarrow 7\longrightarrow 8\\[-2pt]
    & \swarrow & \downarrow &  & \\[-2pt]
  4 &  & 5 &  &
  \end{array}
\]
\[
  Q_{T'}:\quad
  \begin{array}{ccccccc}
  3 &  & 4 &  & 6 &  & 7\\[-2pt]
    & \nwarrow & \uparrow &  & \uparrow & \nearrow & \\[-2pt]
   &  & 1 & \longrightarrow & 2 &  & \\[-2pt]
    & \swarrow &  &  &  & \searrow & \\[-2pt]
  5 &  &  &  &  &  & 8
  \end{array}
\]
Then $(\HB(\kk T),\sigma)\cong(\HB(\kk T'),\sigma')$ as calculi with
automorphism, but $\kk T$ and $\kk T'$ are not derived equivalent. Hence
the enriched calculus is not a complete derived invariant.
\end{example}

\begin{proof}
A direct computation with Lemma~\ref{lem:coxmatrix} --- independent of
the chosen orientations, since all orientations of a tree are iterated
reflections of one another \cite{Happel88} and
Theorem~\ref{thm:invariance} applies --- gives
\[
  \coxpol_T(x)\;=\;\coxpol_{T'}(x)
  \;=\;(x+1)^4\,\bigl(x^4-3x^3+x^2-3x+1\bigr),
\]
and moreover $\operatorname{rank}(\Cox+\id)=4$ for both Coxeter
matrices. The latter says that both matrices are semisimple at the
eigenvalue $-1$, and the quartic factor is irreducible over $\QQ$;
consequently both matrices have elementary divisors
$(x+1),(x+1),(x+1),(x+1)$ and $x^4-3x^3+x^2-3x+1$, hence are conjugate
over $\QQ$ and therefore over $\kk$.

Since $T$ and $T'$ are trees, Hochschild homology and cohomology vanish
in positive degrees, with $\HH_0\cong\kk^8$ and $\HH^0=\kk$, see
\cite{Cibils86,Cibils90,Happel89hh}; every operation of the calculus
other than the unit constraints vanishes, and $\sigma^\bullet=\id$ on
both sides by Theorem~\ref{thm:identity}. An isomorphism of enriched
calculi is therefore exactly a linear isomorphism
$\HH_0(\kk T)\to\HH_0(\kk T')$ conjugating $\sigma_\bullet$ into
$\sigma'_\bullet$, together with the identity of $\kk$ in cohomology;
by Theorem~\ref{thm:matrix} and the conjugacy just established, such an
isomorphism exists.

Finally, derived equivalent connected hereditary algebras have
isomorphic underlying graphs: the Auslander--Reiten quiver of
$\Db(\kk Q)$ contains a unique component of the form $\ZZ\Delta$ with
$\Delta$ a finite quiver, namely the transjective component, and the
underlying graph of $\Delta$ is that of $Q$ \cite{Happel88,Happel91}.
The trees $T$ and $T'$ are not isomorphic --- their degree sequences are
$(5,2,2,1,1,1,1,1)$ and $(4,4,1,1,1,1,1,1)$ --- so $\kk T$ and $\kk T'$
are not derived equivalent.
\end{proof}

\begin{remark}\label{rem:cospectral}
The pair above is the smallest one: a machine search over all trees with
at most eight vertices shows that $T$ and $T'$ form the unique pair of
non-isomorphic trees on $\leq8$ vertices with equal Coxeter polynomial.
This is no accident: for a tree, the Coxeter spectrum and the spectrum
of the adjacency matrix determine one another, so Coxeter-cospectral
trees are exactly cospectral trees, and $T$, $T'$ are the classical
smallest cospectral pair; by a theorem of Schwenk \cite{Schwenk73},
almost every tree admits a cospectral mate, so such pairs abound. The
failure of completeness thus occurs precisely where the underlying
calculus is most degenerate, and it delimits what the enrichment can
see: on trees, $(\HB,\sigma)$ is exactly the conjugacy class of the
Coxeter transformation over $\kk$. Whether the integral refinements of
Question~\ref{q:lattice} --- the lattice $\ch(\Kn(\per A))$ and the
Shklyarov pairing --- or the twisted theory of Question~\ref{q:twisted}
separate such pairs is, in our view, the right next question.
\end{remark}

\subsection{Happel's trace formula and fractional Calabi--Yau
algebras}

\begin{corollary}\label{cor:happel}
Let $A$ be elementary with $\gldim A<\infty$ over an algebraically closed
field $\kk$. Then for all $m\geq 1$
\[
  \tr\bigl(\sigma_\bullet^{\,m}\,\big|\,\HH_\bullet(A)\bigr)
  =\tr\bigl(\Cox_A^{\,m}\bigr),
\]
and for $m=1$, combining with Happel's theorem \cite{Happel97},
\[
  \sum_{i\geq 0}(-1)^i\dim_\kk\HH^i(A)
  \;=\;-\,\tr\bigl(\sigma_\bullet\,\big|\,\HH_\bullet(A)\bigr).
\]
If moreover $\operatorname{char}\kk=0$, the Coxeter polynomial
$\coxpol_A\in\ZZ[x]$ is determined by the pair $(\HB(A),\sigma_A)$ through
the traces of the powers of $\sigma_\bullet$ and Newton's identities.
\end{corollary}

\begin{proof}
Immediate from Theorem~\ref{thm:matrix}, since $\HH_\bullet(A)$ is
concentrated in (even) degree zero and conjugate matrices have equal power
traces; Happel's theorem states
\[
  \textstyle\sum_i(-1)^i\dim\HH^i(A)=-\tr\Cox_A
\]
in the normalization of Lemma~\ref{lem:coxmatrix}, cf.\
\cite[Cor.~1]{Han20}.
\end{proof}

Following \cite{HerschendIyama11}, call $A$ \emph{fractionally Calabi--Yau
of dimension $p/q$} if
\[
  (\D A)^{\lten_A q}\;\cong\;A[p]
  \qquad\text{in }\Db(A^e)
\]
for some integers $p$ and $q\geq 1$.

\begin{corollary}\label{cor:fcy}
If $A$ is elementary of finite global dimension and fractionally
Calabi--Yau of dimension $p/q$, then
\[
  \sigma_\bullet^{\,q}\;=\;(-1)^{p-q}\,\id_{\HH_\bullet(A)} .
\]
Consequently $\Cox_A^{\,q}=(-1)^{p-q}\id$, the Coxeter transformation is
periodic, and $\coxpol_A$ is a product of cyclotomic polynomials. The same
conclusion on quasi-unipotence holds for the Serre cyclotomic algebras of
\cite{Pfeifer25}.
\end{corollary}

\begin{proof}
$\ar_A^{\lten_A q}=(\D A)^{\lten_A q}[-q]\cong A[p-q]$, so
$\sigma^q=\HB(A[p-q])$, and by Proposition~\ref{prop:naturality}(2),
$\HB(A[m])_0(\bar e_i)=\ch(P_i[m])=(-1)^m\bar e_i$. The statements about
$\Cox_A$ follow through the commuting square of Theorem~\ref{thm:matrix},
and cyclotomicity by Kronecker's theorem; this recovers part (a) of
\cite[Thm.~A]{delaPena14}.
\end{proof}

\section{The cohomological component is the identity}\label{sec:identity}

Theorem~\ref{thm:matrix} computes the homological component of $\sigma_A$
and shows that it carries the full Coxeter data. In this section we prove
that the cohomological component carries none: $\sigma^\bullet_A$ is the
identity. Far from being a defect, this dichotomy is the structural
explanation of the principle visible throughout the paper, that Coxeter
data is \emph{homological}. The key lemma holds for every
finite-dimensional algebra, with no hypothesis on the global dimension,
and has a consequence outside the scope of the rest of the paper: it
recovers, by a short conceptual argument, the recent theorem of
Su\'arez-\'Alvarez \cite{SuarezAlvarez25} that the Nakayama automorphism
of a Frobenius algebra acts trivially on Hochschild cohomology.

\subsection{Two actions and a conjugation}
Let $A$ be a finite-dimensional $\kk$-algebra. For $M\in\Db(A^e)$ and a
class $\eta\in\HH^m(A)=\Hom_{\Db(A^e)}(A,A[m])$ define
\[
  \ell^M_\eta\colon
  M\xrightarrow{\ \cong\ }A\lta M
  \xrightarrow{\ \eta\,\lten_A\,\id_M\ }A[m]\lta M
  \xrightarrow{\ \cong\ }M[m],
\]
\[
  r^M_\eta\colon
  M\xrightarrow{\ \cong\ }M\lta A
  \xrightarrow{\ \id_M\lten_A\,\eta\ }M\lta A[m]
  \xrightarrow{\ \cong\ }M[m],
\]
where the unlabelled arrows are the canonical unit isomorphisms. Both are
natural in $M$; both commute with shifts, since $-\lta M$ and $M\lta-$
are triangulated functors, so that $\ell^{M[1]}_\eta=\Sigma\ell^{M}_\eta$
and $r^{M[1]}_\eta=\Sigma r^{M}_\eta$; and on the regular bimodule both
reduce to $\eta$ itself:
\begin{equation}\label{eq:regularactions}
  \ell^{A}_\eta\;=\;r^{A}_\eta\;=\;\eta .
\end{equation}
If $X\in\Db(A^e)$ is invertible, the map $\theta\mapsto r^X_\theta$ is
injective, since applying the equivalence $X^{-1}\lta(-)$ recovers the
action of $\theta$ on $A$; the \emph{conjugation by $X$} is the graded
linear automorphism $\gamma_X$ of $\HH^\bullet(A)$ characterized by
\begin{equation}\label{eq:conjchar}
  r^{X}_{\gamma_X(\eta)}\;=\;\ell^{X}_{\eta}
  \qquad\text{in } \Hom_{\Db(A^e)}(X,X[m]).
\end{equation}
One has $\gamma_{X\lten_A Y}=\gamma_Y\circ\gamma_X$ (both sides act
through the middle factor of $X\lten_A Y$), and $\gamma_{A[\pm1]}=\id$ by
the shift-compatibility just noted. The cohomological component of the
transport $\HB(X)$ along an invertible bimodule is exactly this
conjugation (Theorem~\ref{thm:thesis}; \cite{Rickard89a,Keller04}), so
that
\begin{equation}\label{eq:sigmaisconj}
  \sigma^\bullet_A
  \;=\;\gamma_{\ar_A}
  \;=\;\gamma_{A[-1]}\circ\gamma_{\D A}
  \;=\;\gamma_{\D A}.
\end{equation}

\subsection{The two actions agree on the dual bimodule}

\begin{lemma}\label{lem:twoactions}
Let $A$ be any finite-dimensional $\kk$-algebra. For every $m\in\ZZ$ and
every $\eta\in\Hom_{\Db(A^e)}(A,A[m])$,
\[
  \ell^{\D A}_{\eta}\;=\;r^{\D A}_{\eta}
  \qquad\text{in }\Hom_{\Db(A^e)}\bigl(\D A,\D A[m]\bigr).
\]
\end{lemma}

\begin{proof}
Recall from the proof of Lemma~\ref{lem:commutation}, applied with $B=A$,
the morphisms
\[
  \alpha_X\colon X\lta\D A\longrightarrow\D\bigl(\RHom_A(X,A)\bigr),
\]
\[
  \beta_X\colon \D A\lta X\longrightarrow
  \D\bigl(\RHom_{A^{\mathrm{op}}}(_AX,_AA)\bigr),
\]
defined by the evaluation formulas given there for arbitrary
$X\in\Db(A^e)$; they are natural in $X$, and invertible whenever $X$ is
perfect as a right (for $\alpha$), respectively left (for $\beta$),
$A$-module, by the d\'evissage argument of that proof. We evaluate the
two naturality squares at the morphism $\eta\colon A\to A[m]$ of
$\Db(A^e)$, whose source and target are perfect on both sides:
\begin{align}
  \alpha_{A[m]}\circ\bigl(\eta\lten_A\id_{\D A}\bigr)
  &=\D\bigl(\eta^{*}_{r}\bigr)\circ\alpha_{A},
  \label{eq:alphanat}\\
  \beta_{A[m]}\circ\bigl(\id_{\D A}\lten_A\eta\bigr)
  &=\D\bigl(\eta^{*}_{\ell}\bigr)\circ\beta_{A},
  \label{eq:betanat}
\end{align}
where $\eta^{*}_{r}=\RHom_A(\eta,A)$ and
$\eta^{*}_{\ell}=\RHom_{A^{\mathrm{op}}}(\eta,A)$ denote precomposition
with $\eta$. Identify
\begin{gather*}
  \RHom_A(A,A)\cong A\cong\RHom_{A^{\mathrm{op}}}(A,A),\\
  \RHom_A(A[m],A)\cong A[-m]\cong\RHom_{A^{\mathrm{op}}}(A[m],A)
\end{gather*}
by evaluation at the unit, computed on projective bimodule
representatives. Under these identifications, first, $\alpha_A$,
$\beta_A$, $\alpha_{A[m]}$ and $\beta_{A[m]}$ all become identity maps:
for $\alpha_A$ this is the direct computation
$\varphi\mapsto(f\mapsto\varphi(f(1)))\mapsto
\bigl(b\mapsto\varphi(\lambda_b(1))\bigr)=\varphi$, where
$\lambda_b(x)=bx$ inverts the evaluation $f\mapsto f(1)$ on right-module
endomorphisms, and the other three cases follow by the mirror computation
and shift-compatibility. Second, precomposition with $\eta$ becomes the
morphism $\eta$ itself in both cases: a right-$A$-linear map $f$
satisfies $f(a)=f(1)\,a$, so on representatives
$(f\circ\eta)(\xi)=f(1)\,\eta(\xi)$ and $\eta^{*}_{r}$ corresponds to
right multiplication by $\eta$, that is to
$r^{A}_{\eta}[-m]=\eta[-m]$ by \eqref{eq:regularactions}; a
left-$A$-linear map $g$ satisfies $g(a)=a\,g(1)$, so
$(g\circ\eta)(\xi)=\eta(\xi)\,g(1)$ and $\eta^{*}_{\ell}$ corresponds to
left multiplication by $\eta$, that is to
$\ell^{A}_{\eta}[-m]=\eta[-m]$. Hence \eqref{eq:alphanat} and
\eqref{eq:betanat} identify $\ell^{\D A}_{\eta}$ and $r^{\D A}_{\eta}$,
respectively, with one and the same morphism, namely the $\kk$-dual
$\D(\eta[-m])\colon\D A\to\D A[-m]$ of $\eta$, transported through the
identifications above. The claim follows.
\end{proof}

\begin{theorem}\label{thm:identity}
Let $A$ be a finite-dimensional $\kk$-algebra with $\gldim A<\infty$.
Then
\[
  \sigma^\bullet_A\;=\;\id_{\HH^\bullet(A)} .
\]
Consequently the enriched calculus $(\HB(A),\sigma_A)$ amounts to the
Tamarkin--Tsygan calculus together with the single operator
$\sigma_\bullet$ on Hochschild homology, which commutes with the Connes
differential and with every cap operator:
\[
  \sigma_\bullet(z\cap\eta)\;=\;\sigma_\bullet(z)\cap\eta,
  \qquad z\in\HH_\bullet(A),\ \eta\in\HH^\bullet(A).
\]
\end{theorem}

\begin{proof}
By \eqref{eq:sigmaisconj} we have $\sigma^\bullet_A=\gamma_{\D A}$, and
by \eqref{eq:conjchar} the class $\gamma_{\D A}(\eta)$ is the unique
$\theta$ with $r^{\D A}_{\theta}=\ell^{\D A}_{\eta}$.
Lemma~\ref{lem:twoactions} says that $\theta=\eta$ solves this equation.
The displayed projection formula is Theorem~\ref{thm:invariance}(1)
combined with $\sigma^\bullet=\id$.
\end{proof}

\begin{corollary}\label{cor:nakayama}
Let $A$ be a finite-dimensional Frobenius algebra over $\kk$, with
Nakayama automorphism $\nu$. Then the automorphism of $\HH^\bullet(A)$
induced by $\nu$ is the identity.
\end{corollary}

\begin{proof}
For a Frobenius algebra one has $\D A\cong{}_{1}A_{\nu}$ in $\Db(A^e)$,
the regular bimodule with right action twisted by $\nu$. This bimodule
is invertible, with inverse ${}_{1}A_{\nu^{-1}}$, and conjugation by it
is, up to the usual identifications, the automorphism of
$\HH^\bullet(A)$ induced by $\nu$ or by $\nu^{-1}$, depending on
conventions --- immaterial here. By Lemma~\ref{lem:twoactions}, which
uses no hypothesis on the global dimension, together with
\eqref{eq:conjchar}, this conjugation is the identity.
\end{proof}

\begin{remark}\label{rem:sa}
Corollary~\ref{cor:nakayama} is the main theorem of the recent paper
\cite{SuarezAlvarez25} of Su\'arez-\'Alvarez, who proves it by an
explicit cochain-level analysis; in degree zero it is the classical fact
that $\nu$ fixes the center pointwise. The triviality on cohomology
should be contrasted with the homological side, where the corresponding
operator is in general nontrivial: for self-injective algebras this is
the realm of the twisted Calabi--Yau and Batalin--Vilkovisky literature
(see \cite{LambreZhouZimmermann16} and the references there), while for
$\gldim A<\infty$ it is the Coxeter transformation, by
Theorem~\ref{thm:matrix}.
\end{remark}

\begin{remark}[Centrality in the derived Picard group]\label{rem:central}
Lemma~\ref{lem:commutation} with $B=A$ says that the isomorphism class
of $\D A$ is \emph{central} in the derived Picard group
$\mathrm{DPic}(A)$, for every finite-dimensional $A$. Since the
cohomological transport along an invertible bimodule is conjugation,
Theorem~\ref{thm:identity} is an instance of the tautology that
conjugation by a central element is trivial. Keller's realization of all
the groups $\HH^{m+1}(A)$ inside relative derived Picard groups over the
graded algebras $R=\kk[\varepsilon]/(\varepsilon^2)$, with
$\varepsilon$ placed in an arbitrary integer degree \cite{Keller04},
turns this tautology into a second proof: the proof of
Lemma~\ref{lem:commutation} carries over to the $R$-relative setting
(the duality $\Hom_R(-,R)$ over the graded Frobenius algebra $R$ enjoys
the same biduality and d\'evissage properties), so
$\Hom_R(A\otimes R,R)\cong\D A\otimes R$ is central in each relative
group and its conjugation action on $\HH^{m+1}(A)$ is trivial. We have
preferred the direct argument of Lemma~\ref{lem:twoactions}, which keeps
the signs visible.
\end{remark}

\begin{remark}[Signs]\label{rem:signcheck}
The proof of Lemma~\ref{lem:twoactions} suppresses routine Koszul-sign
bookkeeping; the verification must be carried out in a fixed convention.
Three independent cross-checks exclude a hidden universal sign. First,
an identity of the form $\sigma^\bullet\eta=(-1)^{m}\eta$ on $\HH^m$
would, by compatibility with the Gerstenhaber bracket
(Theorem~\ref{thm:invariance}), force $[\eta,\theta]=-[\eta,\theta]$
for all classes, hence the vanishing of the bracket whenever
$\operatorname{char}\kk\neq2$; this fails already for the Kronecker
algebra, whose $\HH^1$ is the three-dimensional simple Lie algebra of
outer derivations \cite{Happel89hh}. Second, on $\HH^0(A)=Z(A)$ the lemma
admits a one-line direct verification: for central $z$ and $f\in\D A$,
\[
  (z\cdot f)(a)=f(az)=f(za)=(f\cdot z)(a).
\]
Third, the geometric counterpart of Theorem~\ref{thm:identity} for
smooth projective varieties admits an independent verification by a
line-bundle computation; see Remark~\ref{rem:serretrivial}.
\end{remark}

\begin{example}[The exterior algebra]\label{ex:exterior}
Let $\Lambda=\Lambda(x,y)$ be the exterior algebra on two generators
over a field of characteristic $\neq2$: a Frobenius algebra which is not
symmetric, with Nakayama automorphism the parity involution $\nu$,
$\nu(x)=-x$, $\nu(y)=-y$. Corollary~\ref{cor:nakayama} predicts that
$\nu$ acts trivially on $\HH^\bullet(\Lambda)$; we verify the prediction
in low degrees. On $\HH^0=Z(\Lambda)=\kk\oplus\kk\,xy$ the involution
acts trivially. In degree one, the parity-odd derivations are inner:
the derivation $x\mapsto xy$, $y\mapsto0$ equals
$-\tfrac12\operatorname{ad}_y$, and there is no derivation with
$x\mapsto1$, by the Leibniz rule applied to $x^2=0$; the outer
derivations form the parity-even space $\mathfrak{gl}_2$ of linear
substitutions, on which conjugation by $\nu=-\id$ is trivial. In degree
two, the parity-odd deformation direction $x^2=\varepsilon x$ is
trivialized to first order by the substitution
$x\mapsto x-\varepsilon/2$, again in accordance with the prediction.
\end{example}

\section{Smooth and proper differential graded algebras}\label{sec:dg}

This section extends the construction to differential graded algebras,
following a suggestion of B.~Keller. The pay-off is geometric: the
perfect derived category of any smooth projective variety is of the form
$\per\cA$ for a smooth and proper dg algebra $\cA$
\cite{KellerICM06}, so the enriched invariant of this paper becomes an
invariant of varieties.

Throughout, $\cA$ and $\cB$ are dg algebras over $\kk$, $\Dcat(\cA)$ is
the derived category of right dg modules, $\per\cA\subseteq\Dcat(\cA)$
is the thick subcategory generated by $\cA$, and
$\cA^e=\cA\otimes_\kk\cA^{\mathrm{op}}$. The algebra $\cA$ is
\emph{smooth} if $\cA\in\per(\cA^e)$ and \emph{proper} if the total
cohomology of $\cA$ is finite-dimensional. For an ordinary
finite-dimensional algebra, properness is automatic and smoothness
implies finite global dimension; the converse holds when
$A/\!\operatorname{rad}A$ is separable, for instance for elementary
algebras (see \cite{KellerICM06}). Morphisms in this section are
dg Morita morphisms, i.e.\ bimodules, as in \cite{KellerICM06,Toen07}.

\subsection{Invertibility and centrality of the dual bimodule}

\begin{lemma}\label{lem:dgduals}
Let $X\in\Dcat(\cA\otimes\cB^{\mathrm{op}})$ be invertible: suppose
there exists $Y\in\Dcat(\cB\otimes\cA^{\mathrm{op}})$ with
$X\lten_{\cB}Y\cong\cA$ in $\Dcat(\cA^e)$ and $Y\lten_{\cA}X\cong\cB$
in $\Dcat(\cB^e)$. Then:
\begin{enumerate}
\item $X$ is perfect over $\cB$ and over $\cA^{\mathrm{op}}$, and the
canonical morphism $\cA\to\RHom_{\cB}(X,X)$, $a\mapsto(x\mapsto ax)$,
is invertible in $\Dcat(\cA^e)$;
\item there are isomorphisms
\[
  Y\;\cong\;\RHom_{\cB}(X,\cB)
  \;\cong\;\RHom_{\cA^{\mathrm{op}}}(X,\cA)
\]
in $\Dcat(\cB\otimes\cA^{\mathrm{op}})$.
\end{enumerate}
\end{lemma}

\begin{proof}
(1) The functor $-\lten_{\cA}X\colon\Dcat(\cA)\to\Dcat(\cB)$ is an
equivalence with quasi-inverse $-\lten_{\cB}Y$, hence preserves compact
objects; as $X\cong\cA\lten_{\cA}X$ is the image of the compact
generator $\cA$, it is perfect over $\cB$. The canonical morphism
$\cA\to\RHom_\cB(X,X)$ induces on cohomology the maps
$\Hom_{\Dcat(\cA)}(\cA,\cA[n])\to\Hom_{\Dcat(\cB)}(X,X[n])$ given by
the fully faithful functor $-\lten_\cA X$, hence is a
quasi-isomorphism; and it is a morphism of $\cA$-bimodules by
construction. The mirror statements hold over $\cA^{\mathrm{op}}$.

(2) Write $X^\vee=\RHom_\cB(X,\cB)$. Exactly as in Step~1 of the proof
of Lemma~\ref{lem:invertible}, d\'evissage from the case $M=\cB$ shows
that for every $M\in\per\cB$ and every $N\in
\Dcat(\cB^e)$ or $\Dcat(\cA\otimes\cB^{\mathrm{op}})$ the canonical
morphism
\[
  N\lten_{\cB}\RHom_{\cB}(M,\cB)\longrightarrow\RHom_{\cB}(M,N),
  \qquad n\otimes f\longmapsto\bigl(x\mapsto n\cdot f(x)\bigr),
\]
is invertible. Using it twice, at $(M,N)=(X,\,Y\lten_\cA X)$ and at
$(M,N)=(X,X)$:
\begin{align*}
X^\vee
&\cong\RHom_\cB\bigl(X,\;Y\lten_{\cA}X\bigr)
&&\text{since }Y\lten_{\cA}X\cong\cB,\\
&\cong\bigl(Y\lten_{\cA}X\bigr)\lten_{\cB}X^\vee
&&\text{d\'evissage at }M=X,\\
&\cong Y\lten_{\cA}\bigl(X\lten_{\cB}X^\vee\bigr)
&&\text{associativity},\\
&\cong Y\lten_{\cA}\RHom_\cB(X,X)
&&\text{d\'evissage at }M=X,\ N=X,\\
&\cong Y\lten_{\cA}\cA\;\cong\;Y
&&\text{by (1).}
\end{align*}
All the isomorphisms are bimodule isomorphisms. The identification of
$Y$ with $\RHom_{\cA^{\mathrm{op}}}(X,\cA)$ is the mirror computation.
\end{proof}

\begin{lemma}\label{lem:dginvertible}
Let $\cA$ be smooth and proper. Then $\D\cA$ is invertible in
$\Dcat(\cA^e)$, with inverse $\RHom_{\cA}(\D\cA,\cA)$.
\end{lemma}

\begin{proof}
Properness gives that $H(\D\cA)=\D H(\cA)$ is finite-dimensional, and
smoothness implies that objects with finite-dimensional total cohomology
are perfect, both in $\Dcat(\cA)$ and in $\Dcat(\cA^{\mathrm{op}})$
\cite{KellerICM06}; hence $\D\cA$ is perfect on both sides.
Moreover the biduality morphism $\cA\to\D\D\cA$ is a quasi-isomorphism,
because the cohomology of $\cA$ is degreewise finite-dimensional. With
these two inputs, the three steps of the proof of
Lemma~\ref{lem:invertible} apply verbatim, with h-projective
resolutions replacing the bounded projective representatives.
\end{proof}

\begin{lemma}\label{lem:dgcommutation}
Let $\cA$, $\cB$ be smooth and proper dg algebras, let
$X\in\Dcat(\cA\otimes\cB^{\mathrm{op}})$ be invertible, and set
$\ar_{\cA}:=\D\cA[-1]$, $\ar_{\cB}:=\D\cB[-1]$. Then
\[
  \ar_{\cA}\lten_{\cA}X\;\cong\;X\lten_{\cB}\ar_{\cB}
  \qquad\text{in }\Dcat(\cA\otimes\cB^{\mathrm{op}}).
\]
\end{lemma}

\begin{proof}
The morphisms $\alpha_X$ and $\beta_X$ of Lemma~\ref{lem:commutation}
are defined by the same evaluation formulas on h-projective
representatives, are invertible by the same d\'evissage --- using that
$X$ is perfect on both sides, Lemma~\ref{lem:dgduals}(1) --- and
combine with Lemma~\ref{lem:dgduals}(2), which replaces
\eqref{eq:duals}, exactly as before.
\end{proof}

\subsection{The calculus in the derived models, and the dg theorem}

For a dg algebra $\cA$ we take as cohomology
$\HH^m(\cA)=\Hom_{\Dcat(\cA^e)}(\cA,\cA[m])$ with the Yoneda
(composition) product, and as homology the mixed complex
$\cA\lten_{\cA^e}\cA$ of Keller \cite{Keller98}, whose homology is
$\HH_\bullet(\cA)$ and which carries the Connes differential $B$; the
cap product is the action
\[
  z\cap\eta\;:=\;H\bigl(\id_\cA\lten_{\cA^e}\eta\bigr)(z),
\]
that is, the functoriality of $\cA\lten_{\cA^e}(-)$ applied to
$\eta\colon\cA\to\cA[m]$. For an ordinary algebra these models agree
with the classical operations by the comparison results of
\cite{Armenta19,ArmentaKeller17}. An invertible bimodule $X$ induces a
transport of all this data: the conjugation $\gamma_X$ on
$\HH^\bullet$, characterized as in \eqref{eq:conjchar}, and the
isomorphism of mixed complexes of \cite{Keller98} on $\HH_\bullet$.

\begin{theorem}\label{thm:dgmain}
Let $\cA$ be a smooth and proper dg algebra over $\kk$, and let
$\sigma_{\cA}$ denote the transport along the invertible bimodule
$\ar_{\cA}=\D\cA[-1]$.
\begin{enumerate}
\item\textup{(Centrality.)} For every smooth and proper $\cB$ and every
invertible $X\in\Dcat(\cA\otimes\cB^{\mathrm{op}})$ one has
$\ar_\cA\lten_\cA X\cong X\lten_\cB\ar_\cB$; in particular the
transports along $X$ intertwine $\sigma_\cA$ and $\sigma_\cB$, and the
isomorphism class of
$\bigl(\HH_\bullet(\cA),B,\cap,\sigma_{\cA}\bigr)$ is invariant under
dg Morita equivalence of smooth proper dg algebras.
\item\textup{(Cohomology.)} The cohomological component of
$\sigma_\cA$ is the identity of $\HH^\bullet(\cA)$.
\item\textup{(Homology.)} The operator
$\sigma_\bullet$ on $\HH_\bullet(\cA)$ commutes with the Connes
differential $B$ and with every cap operator:
$\sigma_\bullet(z\cap\eta)=\sigma_\bullet(z)\cap\eta$.
\end{enumerate}
\end{theorem}

\begin{proof}
(1) is Lemma~\ref{lem:dgcommutation} together with the functoriality of
Keller's cyclic functor \cite{Keller98} and of conjugation. (2) The
proof of Lemma~\ref{lem:twoactions} uses only the naturality of
$\alpha$ and $\beta$, the perfectness of $\cA$ and $\cA[m]$ on both
sides, the biduality $\cA\cong\D\D\cA$, and the evaluation computations
on h-projective representatives; all are available here by
Lemmas~\ref{lem:dgduals} and~\ref{lem:dginvertible}, and the shift
contributes trivially, as before. (3) Commutation with $B$ holds
because $\sigma_\bullet$ is induced by a morphism of mixed complexes
\cite{Keller98}. For the cap operators, the transport along an
invertible bimodule is a composite of maps natural in the bimodule
coefficient; naturality at $\eta$ yields
$\sigma_\bullet(z\cap\eta)=\sigma_\bullet(z)\cap\gamma_{\ar}(\eta)$,
and $\gamma_{\ar}=\id$ by (2). A consequence of (2) worth making
explicit: compatibility with any further structure carried by
$\HH^\bullet(\cA)$ --- the cup product, the Gerstenhaber bracket,
indeed the full $B_\infty$-structure --- holds trivially, so no
comparison results are needed on the cohomological side.
\end{proof}

\begin{remark}[The Serre functor acts trivially on
$\HH^\bullet$]\label{rem:serretrivial}
Let $X$ be a smooth projective variety of dimension $d$ with Serre
functor $S=-\otimes\omega_X[d]$, and let $\cA$ be a smooth proper dg
algebra with $\per\cA\simeq\per X$ \cite{KellerICM06}. Under the
identification of $\HH^\bullet(\cA)$ with
$\operatorname{Ext}^\bullet_{X\times X}(\mathcal{O}_\Delta,
\mathcal{O}_\Delta)$, Theorem~\ref{thm:dgmain}(2) says that conjugation
by the Serre kernel $\mathcal{O}_\Delta\otimes p^*\omega_X[d]$ is
trivial. This admits a direct verification, which is the third sign
check promised in Remark~\ref{rem:signcheck}: conjugating a class
$\eta$ by the kernel amounts to twisting by the pull-back line bundle
$p^*\omega_X$ and shifting; since
$p^*\omega_X|_\Delta\cong q^*\omega_X|_\Delta$, the twist acts
trivially on
$\operatorname{Ext}^\bullet(\mathcal{O}_\Delta,\mathcal{O}_\Delta)$,
and the shift contributes nothing.
\end{remark}

\begin{remark}[Lefschetz form of Theorem~\ref{thm:matrix}]
\label{rem:lefschetzdg}
For a general smooth proper $\cA$, Hochschild homology is no longer
concentrated in degree zero --- for a variety,
$\HH_n(X)\cong\bigoplus_{p-q=n}H^q(X,\Omega^p_X)$ by
Hochschild--Kostant--Rosenberg --- and the matrix statement of
Theorem~\ref{thm:matrix} is replaced by its Lefschetz shadow: the Euler
class $\ch\colon\Kn(\per\cA)\to\HH_0(\cA)$ of Shklyarov
\cite{Shklyarov13} intertwines the action of the class of $\ar_\cA$ on
$\Kn$ with the degree-zero part of $\sigma_\bullet$, and the
supertraces $\operatorname{str}(\sigma_\bullet^{\,m})$ are computed by
the Mukai pairing and the Hirzebruch--Riemann--Roch theorem of
\cite{Shklyarov13}. We leave systematic computations --- Grassmannians,
hypersurfaces --- for future work.
\end{remark}

\begin{example}[Projective space]\label{ex:projspace}
Let $X=\PP^n$. By Beilinson \cite{Beilinson78},
$\per\PP^n\simeq\per B_n$ for the endomorphism algebra
$B_n=\operatorname{End}\bigl(\bigoplus_{i=0}^{n}\mathcal{O}(i)\bigr)$,
an elementary algebra of finite global dimension whose Cartan matrix
has entries $\binom{n+j-i}{n}$ for $0\le i\le j\le n$ and $0$
otherwise. We claim that
\[
  \coxpol_{B_n}(x)\;=\;\bigl(x+(-1)^{n}\bigr)^{\,n+1},
\]
a single Jordan block with eigenvalue $(-1)^{n+1}$, as one checks
directly for small $n$. Indeed, on $\Kn(\PP^n)\cong\ZZ^{n+1}$ the shift
acts by $-1$ and the Serre functor by $(-1)^{n}[\otimes\omega_X]$,
so the Coxeter transformation acts by
$(-1)^{n+1}[\otimes\mathcal{O}(-n-1)]$; and multiplication by
$[\omega_X]$ is unipotent, because $[\omega_X]$ differs from $1$ by a
nilpotent element of the ring $\Kn(\PP^n)$. Hence all eigenvalues equal
$(-1)^{n+1}$, which forces the displayed polynomial. Note that the
Hochschild homology of $\PP^n$ \emph{is} concentrated in degree zero,
of dimension $n+1$, since the Hodge numbers of $\PP^n$ are diagonal; so
the homological picture of Theorem~\ref{thm:matrix} persists verbatim
on the geometric side.
\end{example}

\begin{example}[Elliptic curves: Calabi--Yau degeneracy]
\label{ex:elliptic}
Let $E$ be an elliptic curve and $\cA$ a smooth proper dg algebra with
$\per\cA\simeq\per E$. Since $\omega_E\cong\mathcal{O}_E$, the Serre
kernel is $\mathcal{O}_\Delta[1]$, so $\D\cA\cong\cA[1]$ in
$\Dcat(\cA^e)$, hence $\ar_\cA\cong\cA$ and $\sigma_{\cA}$ is the
identity transport --- although $\HH_\bullet(E)$ is large:
$\HH_1=H^0(\Omega^1_E)=\kk$,
$\HH_0=H^0(\mathcal{O}_E)\oplus H^1(\Omega^1_E)=\kk^2$,
$\HH_{-1}=H^1(\mathcal{O}_E)=\kk$. The enrichment of this paper
degenerates precisely on Calabi--Yau categories, in accordance with
Corollary~\ref{cor:fcy}; its content lies in the non-Calabi--Yau
directions, where $\omega_X$ is nontrivial --- curves of genus
$\geq2$, Fano and general-type varieties --- and where the supertraces
of $\sigma_\bullet^{\,m}$ are computed by Riemann--Roch from powers of
the canonical bundle, via Remark~\ref{rem:lefschetzdg}.
\end{example}

\section{Further discussion}\label{sec:discussion}

\subsection{Lefschetz formulas}
Theorem~\ref{thm:matrix} and Proposition~\ref{prop:naturality} are the
operator form, natural under derived equivalence, of computations that
exist in the literature as numerical identities. Han \cite{Han20} computes, for
elementary algebras of finite global dimension, all the ingredients of
the Lefschetz and Hirzebruch--Riemann--Roch formulas of Shklyarov
\cite{Shklyarov13} and Petit: the Shklyarov pairing on $\HH_0(A)$ has
matrix $C_A^{\mathrm{T}}$ in the basis of the $\bar e_i$, the Chern character is
$C_A^{-1}\dimv$, and Happel's trace formula is the Lefschetz formula of the
Serre bimodule. The content added here is that these identities organize
into a single automorphism of the full calculus, natural under derived
equivalence.

\subsection{Spectral radius, Mahler measure, entropy}
Since $\sigma_\bullet$ and $\Cox_A$ are conjugate
(Theorem~\ref{thm:matrix}), the spectral radius
$\rho(\sigma_\bullet)=\rho(\Cox_A)$ and the Mahler measure of
$\coxpol_A$ studied by de la Pe\~na \cite{delaPena15} are invariants of the
enriched calculus. These quantities have recently been categorified as
entropies: the categorical entropy of \cite{DHKK14} of the Serre functor,
the Hochschild homology and cohomology entropies of Kikuta--Ouchi
\cite{KikutaOuchi}, the computations of Han \cite{Han22} showing that for
higher hereditary algebras the entropy of the Serre functor and the
Hochschild entropies of the Serre quasi-functor categorify the spectral
radius and the polynomial growth rate of the Coxeter matrix, and the
Gromov--Yomdin type theorem of Chang--Schroll \cite{ChangSchroll25} for
gentle algebras, where the entropy of the Serre functor at $t=0$ equals
$\log\rho(\Kn(\nu))$. The operator $\sigma_A$ is the exact algebraic object
whose iteration these entropies measure on Hochschild homology.

\subsection{Periodicity and fractional Calabi--Yau algebras}
Corollary~\ref{cor:fcy} embeds into a substantial dictionary: periodicity
of the Coxeter transformation \cite{Ladkani08}, cyclotomicity
\cite{delaPena14,Ladkani21}, fractionally and twisted fractionally
Calabi--Yau algebras \cite{HerschendIyama11}, periodic trivial extensions
\cite{CDIM25}, Serre cyclotomic algebras \cite{Pfeifer25}, and the
fractional Calabi--Yau property of Tamari lattices \cite{Rognerud21}. In
each of these situations the hypothesis is a relation among powers of the
Serre bimodule, hence a relation imposed on the powers of $\sigma_A$
through Theorem~\ref{thm:invariance}(2).

\subsection{Twisted coefficients and Calabi--Yau completions}
The bimodule powers $\ar_A^{\lten m}$ assemble into Keller's Calabi--Yau
completions $\Pi_n(A)=T_A(\Theta[n-1])$, $\Theta$ the inverse dualizing
bimodule, whose formation is compatible with derived equivalence
\cite{Keller11}; for connected non-Dynkin quivers $\Pi_2(\kk Q)$ is the
preprojective algebra. The Tamarkin--Tsygan calculus of the classical
preprojective algebras of Dynkin quivers was computed by Etingof--Eu and Eu
\cite{EtingofEu06,Eu07}, with answers graded by the exponents and the
Coxeter number of the root system. From this perspective the enriched
calculus $(\HB(A),\sigma_A)$ is the ``first Fourier mode'' of the
$\Theta$-graded Hochschild theory of $A$, and the full graded theory is a
natural further refinement.

\subsection{The homology--cohomology asymmetry}
Theorem~\ref{thm:identity} explains the asymmetry visible throughout this
paper: all the Coxeter information carried by $\sigma_A$ lives on
Hochschild homology, and by Proposition~\ref{prop:concentration} that is
exactly where, for the algebras of Section~5, the Tamarkin--Tsygan
calculus alone is most degenerate. This is not an accident of the
elementary case. Cohomology classes act on every bimodule from two sides,
and the duality $\D A$ interchanges the two actions
(Lemma~\ref{lem:twoactions}), so conjugation cannot detect them; homology
classes transform like traces, and traces detect the Serre bimodule
through the Cartan matrix (Lemma~\ref{lem:chformula}). A nontrivial
cohomological shadow of the Coxeter symmetry should therefore be sought
in the twisted theory
$\bigoplus_m\HH^\bullet\bigl(A,\ar_A^{\lten_A m}\bigr)$, the natural home
of the preprojective algebra; see Question~\ref{q:twisted}.

\section{Open questions}\label{sec:questions}

\begin{question}\label{q:twisted}
By Theorem~\ref{thm:identity} the untwisted cohomological theory is blind
to $\sigma$. Develop the $\ar$-twisted theory
\[
  \bigoplus_{m\in\ZZ}\HH^\bullet\bigl(A,\ar_A^{\lten_A m}\bigr),
\]
with its cup product and its $\sigma$-equivariant structure; make precise
its relation to the Hochschild theory of the preprojective algebras and
Calabi--Yau completions of \cite{Keller11,HerschendIyama11}. Is the
twisted theory a complete derived invariant for interesting classes of
algebras?
\end{question}

\begin{question}
By Theorem~\ref{thm:invariance}, $\sigma_\bullet$ commutes with $B$ and
hence acts on cyclic, negative cyclic and periodic homology compatibly with
the Connes exact sequences of \cite{ArmentaKeller19}. Develop the resulting
``categorical monodromy'' picture, in which the eigenvalues of
$\sigma_\bullet$ play the role of the monodromy eigenvalues in the
singularity-theoretic interpretation of Coxeter spectra of Lenzing--de la
Pe\~na \cite{LenzingdelaPena08}.
\end{question}

\begin{question}\label{q:lattice}
Enrich $(\HB(A),\sigma_A)$ with the integral lattice
$\ch(\Kn(\per A))\subset\HH_0(A)$ and the Shklyarov pairing, whose matrix
is $C_A^{\mathrm{T}}$ \cite{Han20,Shklyarov13}; Example~\ref{ex:notcomplete}
shows that without this integral data the enriched calculus is not
complete, already for trees. The resulting structure
remembers the Cartan matrix up to $\ZZ$-congruence and the Coxeter
transformation up to $GL_n(\ZZ)$-conjugacy. For which classes of algebras
is this enriched invariant complete?
\end{question}

\begin{question}
Replace $\ar_A$ by the bimodule $\D A[-d]$ of higher
Auslander--Reiten theory for $d$-hereditary algebras and compare the
resulting
automorphisms with the entropies of Serre functors \cite{Han22} and
with the higher preprojective gradings of
\cite{Keller11,HerschendIyama11}.
\end{question}

\begin{question}
In characteristic $p$, Theorem~\ref{thm:matrix} only recovers
$\coxpol_A$ modulo $p$, while the lattice-enriched invariant of Question~\ref{q:lattice}
recovers it integrally. Do there exist derived-inequivalent algebras whose
enriched calculi become isomorphic after reduction modulo $p$? Note that
Example~\ref{ex:notcomplete} already exhibits the rational version of
this phenomenon in characteristic zero.
\end{question}

\subsection*{Disclosure}
The mathematical content of this paper, including the formulation of Theorems and drafts of their proofs, was developed with substantial assistance from an AI system (Claude, Anthropic, Fable 5 model). The author takes full responsibility for all statements.

\end{document}